\documentclass[12pt, oneside]{article}
\usepackage[english]{babel}
\usepackage[T1]{fontenc}

\usepackage{amssymb}
\usepackage{amsmath}
\usepackage{hhline}
\usepackage{longtable}
\usepackage{amscd}
\usepackage{theorem}
\usepackage{array}
\usepackage{delarray}
\usepackage{multicol}
\usepackage{makeidx}

\usepackage{float}

\usepackage[pdftex]{graphicx}

\newtheorem{theorem}{Theorem}[section]

\newtheorem{corollary}{Corollary}[section]

\newtheorem{definition}{Definition}[section]

\newtheorem{lemma}{Lemma}[section]

\newenvironment{proof}[1][Proof]{\textbf{#1.} }{\ \rule{0.5em}{0.5em}}

\newcommand{\abs}[1]{\left\lvert#1\right\rvert}
\newcommand{\norm}[1]{\parallel #1 \parallel}
\newcommand{\C}{\mathbb{C}}
\newcommand{\Q}{\mathbb{Q}}
\renewcommand{\P}{\mathbb{P}}
\newcommand{\R}{\mathbb{R}}
\newcommand{\Z}{\mathbb{Z}}
\newcommand{\K}{\mathbb{K}}
\newcommand{\N}{\mathbb{N}}
\newcommand{\po}{\textrm{P}}
\newcommand{\Dp}[2]{\dfrac{\partial {#1}}{\partial{#2}}}
\newcommand{\var}{\varepsilon}
\newcommand{\lda}{\lambda}
\newcommand{\bs}{\bigskip}
\newcommand{\ms}{\medskip}
\newcommand{\sk}{\smallskip}
\newcommand{\hr}{\hrulefill}
\newcommand{\Dint}{\displaystyle{\int}}
\newcommand{\dint}{\displaystyle{\int}}
\newcommand{\vsq}{\vspace*{-0.325cm}}
\newcommand{\Inf}[1]{\displaystyle{\inf_{#1}}}
\newcommand{\Sup}[1]{\displaystyle{\sup_{#1}}}
\newcommand{\Max}[1]{\displaystyle{\max_{#1}}}
\newcommand{\Min}[1]{\displaystyle{\min_{#1}}}
\newcommand{\Sum}[2]{\displaystyle{\sum_{#1}^{#2}}}
\newcommand{\Int}[2]{\displaystyle{\int_{#1}^{#2}}}
\newcommand{\IInt}[1]{\displaystyle{\iint_{\!#1}}}
\newcommand{\dd}{\textrm{\hspace{0.05cm}d}}
\newcommand{\ce}{\mbox{\textrm{\, c}}}
\newcommand{\ov}[1]{\overrightarrow{#1\ }}
\newcommand{\Prod}[2]{\displaystyle{\prod_{#1}^{#2}}}
\newcommand{\absolue}[1]{\left| #1 \right|}
\newcommand{\Dlim}[2]{\displaystyle{\lim_{#1\rightarrow #2}}\,}
\newcommand{\di}{\displaystyle{\lim_{n\rightarrow +\infty}}\:}
\newcommand{\lz}{\displaystyle{\lim_{x\rightarrow 0}} \ }
\newcommand{\fonc}[5]{#1 : \begin{cases}#2 \rightarrow #3 \\ #4 \mapsto #5
 \end{cases}}
\newtheorem{PB}{Probl\`{e}me}

\newtheorem{EX}{Exercise}

\newtheorem{EXC}{Exercice}

\newtheorem{so}{Exercice}

\newtheorem{sop}{Probl\`{e}me}

\begin{document}
\begin{center}
Labeled Packing of Non Star Tree into its Fifth Power and Sixth
Power
\end{center}
\begin{center} Amine El Sahili - Hamamache Kheddouci - Maidoun
Mortada\end{center}
\begin{abstract}
 In this paper we prove that we can find a labeled packing of a non star tree $T$ into  $T^6$ with
  $m_T+\lceil\frac{n-m_T}{5}\rceil$ labels, where $n$ is the number of vertices of $T$ and $m_T$ is the maximum
 number of leaves that can be removed from $T$ in such a way that the obtained graph is a non star tree. Also, we prove that we can
 find a labeled packing of a non star tree $T$ into  $T^5$ with $m_T+1$
 labels and a labeled packing of
 a path $P_n$, $n\geq 4$, into $P_n^4$ with $\lceil \frac{n}{4}\rceil$
 labels.\end{abstract}
\section{Introduction}
\indent All graphs considered in this paper are finite and
undirected. For a graph $G$,  $V(G)$ and $E(G)$ will denote its
vertex set and edge set, respectively. We denote by $N_G(x)$ the set
of the neighbors of the vertex $x$ in G.
 The degree $d_G(x)$ of the vertex $x$ in $G$ is the cardinality of the
set $N_G(x)$. For short, we use $d(x)$ instead of $d_G(x)$ and
$N(x)$ instead of $N_G(x)$. The distance between two vertices of
$G$, say $x$ and $y$, is denoted by $dist_G(x,y)$, and for short we
usually use $dist(x,y)$.  For a subset $U$ of $V$, we denote by
$G-U$ the graph  obtained from $G$ by deleting all the vertices in
$U\cap V$ and their incident edges. For a subset $F$ of $E$, we
write $G-F:=(V,E\setminus F)$.\\\\ A vertex of degree one in a tree
$T$ is called a leaf and the neighbor of a leaf is its father. For a
non star tree $T$, we denote by $m_T$ the maximum number of leaves
that can be removed from $T$ in such a way that the obtained tree is
a non star one. The number of edges of a path $P$ is its length
$l(P)$. A path on $n$ vertices is denoted by $P_n$. The middle
vertex of $P_5$ will be called a bad vertex.
\\
\\
\noindent Let $G$ be a graph of order $n$. Consider a permutation
$\sigma: V (G)\rightarrow V (K_n)$, the map $\sigma^*:
E(G)\rightarrow E(K_n)$ such that $\sigma^*(xy)=\sigma(x)\sigma(y)$
is the map induced by $\sigma$. We say that there is a packing of
$k$ copies of $G$ (into the complete graph $K_n$) if there exist
permutations $\sigma_i : V (G)\rightarrow V (K_n)$, where $i =
1,..., k$, such that $\sigma^*_i (E(G)) \cap \sigma^*_j (E(G))
=\phi$ for $i \neq j$. A packing of $k$ copies of a graph $G$ will
be called a $k$-placement of $G$. A packing of two copies of $G$
(i.e. a 2-placement) is also called an embedding of $G$ (into its
complement $\bar{G}$). That is, we say that $G$ can be embedded in
its complement if there exists a permutation $\sigma$ on $V (G)$
such that if an edge $xy$ belongs to $E(G)$, then
$\sigma(x)\sigma(y)$ does not belong to $E(G).$
 A permutation
$\sigma$ on $V(G) $ such that $\sigma(x)\neq x$ for every $x$ in $V
(G)$ is called a
fixed point free permutation. \\
\\ \noindent The problem of embedding paths
and trees in their complements has long been one of the fundamental
questions in combinatorics that has been considerably investigated
[2,4,5,6,7,8]. For recent results and survey on this field, we refer
to the survey
papers of Wozniak [9] and Yap [10].\\
Concerning non star trees, the following theorem was proved by
Straight (unpublished, cf. [3])
\begin{theorem}
Let $T$ be a non star tree, then $T$ is contained in its own
complement.
\end{theorem} This result has been improved in many ways especially in considering some additional
information and conditions about embedding. An example of such a
result is the following theorem contained as a lemma in [9]:
 \begin{theorem} Let $T$ be a non-star tree of order $n$
with $n>3$. Then there exists a 2-placement $\sigma$ of $T$ such
that for every $x \in  V(T)$, $dist (x, \sigma(x))\leq 3$.
\end{theorem}
 This theorem immediately implies the following:
\begin{corollary}
 Let $T$ be a non-star tree of order $n$ with $n>3$. Then
there exists an embedding of  $T$ such that $\sigma(T)\subseteq
T^7$.
\end{corollary}
In [6], Kheddouci \emph{et al.} gave a better improvement in the
following theorem:
\begin{theorem}
Let $T$ be a non star tree and let $x$ be a vertex of $T$. Then,
there exists a permutation $\sigma$ on $V(T)$ satisfying the
following four conditions: \\
1. $\sigma$ is a 2-placement of $T$.\\
2. $\sigma(T)\subseteq T^4$.\\
3. $dist(x,\sigma(x)) = 1$.\\
4. for every  neighbor $y$ of $x$, $dist(y,\sigma(y))\leq 2$.
\end{theorem}
Labeled graph packing is a well known field of graph theory that has
been considerably investigated. It is intorduced by E. Duchene and
H. Kheddouci. Below is the definition of the labeled packing
problem:
\begin{definition}. Consider a graph $G$. Let $f$ be a mapping from $V (G)$ to the set $\{1, 2,... ,
p\}$. The mapping $f$ is called a p-labeled-packing of $k$ copies of
$G$ into $K_n$ if there exists  permutations $\sigma_i : V (G)
\rightarrow V(K_n)$, where $i = 1,..., k$, such that:\\
1. $\sigma^*_i (E(G)) \cap \sigma^*_j (E(G))=\phi$ ; for all $i \neq j$.\\
2. For every vertex $v$ of $G$, we have
$f(v)=f(\sigma_1(v))=f(\sigma_2(v))=...=f(\sigma_k(v))$. \\
\end{definition}
The maximum positive integer $p$ for which $G$ admits a
$p$-labeled-packing of $k$ copies of $G$ is called the labeled
packing number of $k$ copies of $G$ and is denoted by
$\lambda^k(G)$.  \\\\E. Duchene \emph{et al.} introduced the
following two results which are presented as Lemmas in [1]. These
results give an upper bound for the labeled packing $\lambda^2(G)$.
\begin{theorem}
Let $G$ be a graph of order $n$ and let $I$ be a maximum independent
set of $G$. If there exists an embedding of $G$ into $K_n$, then
\begin{center}
$\lambda^2(G) \leq |I| + \lfloor\frac{ n - |I|}{ 2} \rfloor$
\end{center}
\end{theorem}
\begin{theorem}
Let $G$ be a graph of order  $n$ with a maximum independent set $I$
of size at least $\lfloor\frac{n}{ 2} \rfloor$. If there exists a
packing of $k \geq 2$ copies of $G$ into $K_n$, then \begin{center}
$\lambda^k(G) \leq |I| + \lfloor \frac{n - |I|}{ k} \rfloor$
\end{center}
\end{theorem}
E. Duchene \emph{et al.} $[1]$ introduced and studied the labeled
graph packing problem for some vertex labeled graphs. In this paper,
we are concerned with finding a p-labeled-packing of $G$ into $G^k$.
We give below the definition of the new labeled packing problem:
\begin{definition}
Let $f$ be a mapping from $V (G)$ to the set $\{1, 2,... , p\}$. The
mapping $f$ is called p-labeled-packing of  $G$ into $G^k$ if there
exists a permutation $\sigma : V (G) \rightarrow V
(K_n)$, such that:\\
1. $\sigma $ is a 2-placement of $G$.\\
2. $\sigma(G)\subseteq G^k$.\\
3. For every vertex $v$ of $G$, we
have $f(v)=f(\sigma(v))$. \\
\end{definition}
The maximum positive integer $p$ for which $G$ admits a
$p$-labeled-packing of $G$ into $G^k$ is called labeled packing $k$-power number and denoted by $w^k(G)$.\\\\
Concerning the  packing of a path $P_n$, $n\geq 4$, into $P^4_n$,
 we introduce the following result:
\begin{theorem}
Consider a path $P_n$, $n\geq 4$, and let $u$ and $v$ be its end
vertices. Then there exists a permutation $\sigma$ on $V(P_n)$ such
that $\sigma$ satisfies the following
conditions:\\
1. $\sigma$ is a 2-placement.\\
2. $\sigma(P_n)\subseteq P_n^4$.\\
3. $ dist(u,\sigma(u))=1$ and $dist(v,\sigma(v))\leq 1$. \\
4. The length of each cycle of $\sigma$ is at most 4.
\end{theorem}
This result allows us to establish the following:
\begin{corollary}
Consider a path $P_n$, $n\geq 4$, then $w^4(P_n)\geq \lceil
\frac{n}{4}\rceil$.
\end{corollary}
To formulate our main results we need to introduce some definitions.\\
 Let $T$ be a non star tree and let $x$ be a
vertex of $T$. Then, a fixed point free permutation $\sigma$ on
$V(T)$ is called a $(T,x)$-well
2-placement if it satisfies the following conditions: \\
1. $\sigma$ is a 2-placement of $T$.\\
2. $\sigma(T)\subseteq T^6$.\\
3. $dist(x, \sigma(x)) \leq 2$.\\
4. $dist(y, \sigma(y))\leq 3$ for every neighbor $y$ of $x$.\\
5. $dist(y, \sigma(y))\leq 4$ for every $y$ such that $ d(y) = 1$.\\
6. The length of each cycle of $\sigma$ is at most 5.\\\\ We prove
first:
\begin{theorem}
 Let $T$ be a non star tree and let $x$ be a vertex of $T$. Then,
there exists  a $(T,x)$-well 2-placement.
\end{theorem}
This  implies the following:
\begin{corollary}
 Consider a non star tree $T$ with $|V (T)| = n$, then
$w^6(T) \geq m_T + \lceil\frac{n-m_T}{5}\rceil$.
\end{corollary}
Let $T$ be a non star tree and let $x$ be a vertex of $T$. Then, a
fixed point free permutation $\sigma$ on $V(T)$ is called a
$(T,x)$-good 2-placement
if it satisfies the following conditions: \\
1. $\sigma$ is a 2-placement of $T$.\\
2. $\sigma(T)\subseteq T^5$.\\
3. $dist(x, \sigma(x)) = 1$.\\
4. $dist(y, \sigma(y))\leq 2$ for every neighbor $y$ of $x$.\\
5. $dist(y, \sigma(y))\leq 4$ for every $ y$ such that $ d(y) =
1$.\\\\
 We prove
that:
 \begin{theorem}
 Let $T$ be a non star tree and let $x$ be a vertex of $T$ such that $x$ is not a bad vertex. Then,
 there exists a $(T,x)$-good 2-placement.
\end{theorem}
This result allows us to establish the following:
\begin{corollary}
 Consider a non star tree $T$, then
$w^5(T) \geq m_T + 1$.
\end{corollary}

\section{Labeled Packing of $P_n$ into $P_n^4$}
In this section, we are going to prove Theorem 1.6, but  we need
first to prove the theorem for $n=4,...,7$:
\begin{lemma}
Consider a path $P_n$ such that $4\leq n\leq 7$, and let $u$ and $v$
be its end vertices. Then there exists a permutation $\sigma$ on
$V(P_n)$ satisfying the following conditions:
\\
1. $\sigma$ is a 2-placement.\\
2. $\sigma(P_n)\subseteq P_n^4$.\\
3.  $ dist(u,\sigma(u))=1$ and $dist(v,\sigma(v))\leq 1$. \\
4. The length of each cycle of $\sigma$ is at most 4.
\end{lemma}
\begin{proof}
For each path $P_n$, $n=4,...,7$, we will introduce below a
permutation
$\sigma$ on $V(P_n)$, satisfying the above conditions:\\
For $P=x_1x_2x_3x_4$, $\sigma=(x_1\;x_2\;x_4\;x_3)$.\\
For $P=x_1x_2x_3x_4x_5$,
$\sigma=(x_1\;x_2\;x_5\;x_4)(x_3)$.\\
For $P=x_1x_2x_3x_4x_5x_6$,
$\sigma=(x_1\;x_2\;x_5\;x_4)(x_3)(x_6)$.\\
For $P=x_1x_2x_3x_4x_5x_6x_7$,
$\sigma=(x_1\;x_2\;x_5)(x_3\;x_7\;x_6)(x_4)$.\\
\end{proof}\\\\
\textbf{Proof of Theorem 1.4}.\\ The proof is by induction. By the
previous Lemma, $\sigma$ exists for $n=4,...,7$. Suppose now that
$n\geq 8$ and the theorem holds for all $n'<n$. Then $P_n$ can be
partitioned into two paths $P'$ and $P''$ such that
$l(P'),\;l(P'')\geq 3$. Let $x$ be the end vertex of $P'$ and $y$
that of $P''$ such that $d_{P_n}(x)=d_{P_n}(y)=2$ and let $\sigma'$
and $\sigma''$ be two permutations defined on $V(P' )$ and $V(P'')$
respectively such that $\sigma'$ and $\sigma''$ satisfy the four
conditions mentioned in the theorem with $dist(x,\sigma'(x))=1$ and
$dist(y,\sigma''(y))\leq 1$. Let $\sigma$ be a permutation defined
on $V(P_n)$, such that:
\begin{center}
 $\sigma(v) = \begin{cases}
  \sigma'(v) & \text{if $v\in V(P')$} \\
  \sigma''(v) & \text{if $v\in  V(P'')$} \\
\end{cases}$
\end{center}
It can be easily shown that $\sigma$ satisfies the four conditions. $\square$\\\\
\textbf{Proof of Corollary 1.2}.\\
Consider a path $P_n$, $n\geq 4$, and let $u$ and $v$ be its end
vertices. Then, by the previous theorem there exists a permutation
$\sigma$ on $V(P_n)$ satisfying the following conditions:
\\
1. $\sigma$ is a 2-placement.\\
2. $\sigma(P_n)\subseteq P_n^4$.\\
3. $ dist(u,\sigma(u))=1$ and $dist(v,\sigma(v))\leq 1$. \\ \\
4. The length of each cycle of $\sigma$ is at most 4.\\\\
Let $r$ be the number of cycles of $\sigma$ and let
$\sigma_1$,...,$\sigma_r$ be these cycles. Note that $r\geq
\lceil\frac{n}{4}\rceil$. Label the vertices of $\sigma_i$ by $i$
for $i=1,...,r$. Hence, we obtain a labeled packing of $P_n$ into
$P_n^4$ with $r$ labels and so
$w^4(P_n)\geq\lceil\frac{n}{4}\rceil$. $\square$

\section{Labeled Packing of a Non Star Tree $T$ into $T^5$ and $T^6$}
In this section, we are going to prove the main results of this
paper, but we still need to introduce some definitions and results
on paths followed by a sequence of lemmas.\\ Consider a path $P_n$,
$n\geq 4$, and let $x$ be a vertex of $P_n$. A fixed point free
permutation $\sigma$ on $V(P_n)$ is called a $(P_n,x)$-well path
2-placement, if it satisfies the following conditions:
\\
1. $\sigma$ is a 2-placement of $P_n$. \\
2. $\sigma(P_n)\subseteq P_n^6$.\\
3. $dist(x,\sigma(x))\leq 2$.\\
4. $dist(y,\sigma(y))\leq 3$ for every  $y\in N(x)$ and for every $y$ such that $d(y)=1$.\\
5. The length of each cycle of $\sigma$ is at most 5.\\\\
We will prove the following theorem:
\begin{theorem}
Consider a path $P_n$ and let $x$ be a  vertex of $P_n$, $n\geq 4$.
Then there exists a $(P_n,x)$-well path 2-placement.
\end{theorem}
\begin{lemma}
Consider a path $P_n,\; 4\leq n\leq 7$, and let $x$ be a vertex of
$P_n$. Then there exists a $(P_n,x)$-well path 2-placement, say
$\sigma$, such that $dist(v,\sigma(v))\leq 3$ for every $v \in
V(P_n)$.
\end{lemma}
\begin{proof}
For each path $P_n$, $n=4,...,7$, and for every vertex $x$ of $P_n$
we will introduce below a $(P_n,x)$-well path 2-placement $\sigma$
such that $dist(v,\sigma(v))\leq 3$ for every $ v \in V(P_n)$:\\\\
For $P=x_1x_2x_3x_4$, $\sigma=(x_1\;x_2\;x_4\;x_3)$ is a  $(P,x)$-well path 2-placement for every $ x\in V(P)$.\\
For $P=x_1x_2x_3x_4x_5$, $\sigma=(x_1\;x_2\;x_4\;x_5\;x_3)$
is a $(P,x)$-well path 2-placement  for every $ x\in V(P)$. \\
For $P=x_1x_2x_3x_4x_5x_6$, there are three choices for choosing
$x$, either $x_1$, $x_2$ or $x_3$.
$\sigma=(x_1\;x_2\;x_4)(x_3\;x_6\;x_5)$ is a $(P,x)$-well path
2-placement for $x\in\{x_1,x_2\}$, and
$\sigma=(x_3\;x_1)(x_5\;x_2)(x_6\;x_4)$ is a $(P,x_3)$-well path
2-placement.\\ For $P=x_1x_2x_3x_4x_5x_6x_7$, there are four choices
for choosing $x$, either $x_1$, $x_6$, $x_3$ or $x_4$.
$\sigma=(x_1\;x_2\;x_5\;x_3)(x_4\;x_6\;x_7)$ is a $(P,x)$-well path
2-placement for every $x$ of the previous choices.
\end{proof}
\\\\
\textbf{ Proof of Theorem 3.1}.\\ The proof is by induction. By the
previous Lemma, there exists a $(P_n,x)$-well path 2-placement
 for every vertex $x$ of $P_n$, where $4\leq n\leq 7$.
Suppose now $n\geq 8$ and the theorem holds for all $n'<n $. Let $x$
be a vertex of $P_n$. Since $n\geq 8$, then $P_n$ can be partitioned
into two paths $P'$ and $P''$ such that $l(P'),\;l(P'')\geq 3$.
Without loss of generality, suppose that $x\in V(P')$. Let $x_1$ be
the end vertex of $P''$ such that $d_{P_n}(x_1)=2$. By induction,
there exists a $(P',x)$-well path 2-placement, say $\sigma_x$, and
there exists a $(P'',x_1)$-well path 2-placement, say
$\sigma_{x_1}$. Let $\sigma$ be a permutation defined on $V(P_n)$
such that:
\begin{center}
 $\sigma(v) = \begin{cases}
  \sigma_x(v) & \text{if $v\in V(P')$} \\
  \sigma_{x_1}(v) & \text{if $v\in  V(P'')$} \\
\end{cases}$
\end{center}
It can be easily shown  that $\sigma$ is a $(P_n,x)$-well path
2-placement. $\square$\\\\
Consider a path $P_n$, $n\geq 4$, and let $x$ be a vertex of $P_n$.
We say that a fixed point free permutation $\sigma$ on $V(P_n)$ is a
$(P_n,x)$-good path 2-placement if $\sigma$ satisfies the following
conditions:
\\
1. $\sigma$ is a 2-placement of $P_n$. \\
2. $\sigma(P_n) \subseteq P_n^5$.\\
3. $dist (x,\sigma(x))=1$.\\
4. $dist(y, \sigma(y)) \leq 2$ for every $y\in N(x)$ and for every $y$ such that $d(y)=1$.\\
\\
We  prove:
\begin{theorem}
Let $x$ be a vertex of  $P_n,\; n\geq 4,$  such that $x$ is not a
bad vertex, then there exists a $(P_n,x)$-good path 2-placement.
\end{theorem}

 \begin{lemma}
For every $x$ in $V(P_n)$, $4\leq n\leq 7$, such that $x$ is not a
bad vertex, there exists a $(P_n,x)$-good path 2-placement.
\end{lemma}
\begin{proof}
For every x in $V (P_n)$, $n=4,...,7$, we will introduce below a
$(P_n, x)$-good path 2-placement $\sigma_x$:\\\\
For $P=x_1x_2x_3x_4$, there are  two choices for choosing $x$,
either $x_1$ or $x_2$. Then:
\begin{center}
$\sigma_{x_1}=(x_1\;x_2\;x_4\;x_3$);
$\sigma_{x_2}=(x_1\;x_3\;x_4\;x_2)$
\end{center}
For $P=x_1x_2x_3x_4x_5$, there are two choices for choosing $x$,
either $x_1$ or $x_4$. Then: \begin{center}
$\sigma_{x_1}=\sigma_{x_4}=(x_1\;x_2\;x_4\;x_5\;x_3)$.
\end{center}
For $P=x_1x_2x_3x_4x_5x_6$, there are  three  choices for
choosing $x$, either $x_1$, $x_2$ or $x_3$. Then:\begin{center} $
\sigma_{x_1}=(x_1\;x_2\;x_4\;x_3\;x_6\;x_5);\;\sigma_{x_2}=\sigma_{x_3}=(x_1\;x_3\;x_4\;x_6\;x_5\;x_2)$.
\end{center}
For $P=x_1x_2x_3x_4x_5x_6x_7$, there are  four  choices for choosing
$x$, either $x_1$,  $x_2$,  $x_3$ or $x_4$. Then:
\begin{center}
$\sigma_{x_1}=
\sigma_{x_4}=(x_1\;x_2\;x_4\;x_5\;x_7\;x_6\;x_3)$;
$\sigma_{x_2}=\sigma_{x_3}=(x_1\;x_3\;x_4\;x_6\;x_7\;x_5\;x_2)$.
\end{center}
\end{proof}

\noindent\textbf{Proof of Theorem 3.2}.\\ The proof is by induction.
By the previous Lemma, there exists a $(P_n,x)$-good path
2-placement for every $ x\in V(P_n)$ such that $x$ is not a bad
vertex, where $4\leq n\leq 7$. Suppose now $n\geq 8$ and the theorem
holds for all $n'< n$. Let $x$ be a vertex of $P_n$. Since $n\geq
8$, then $P_n$ can be partitioned into two paths $P'$ and $P''$ such
that $l(P'),\;l(P'')\geq 3$. Without loss of generality, suppose
that $x\in V(P')$ such that $P'$ is chosen to be distinct from
$P_5$. Let $x_1$ be the end vertex of $P''$ such that
$d_{P_n}(x_1)=2$. By induction, there exists a $(P',x)$-good path
2-placement, say $\sigma_x$, and there exists a $(P'',x_1)$-good
path 2-placement, say $\sigma_{x_1}$. Let $\sigma$ be a permutation
defined on $V(P_n)$ such that:
\begin{center}
 $\sigma(v) = \begin{cases}
  \sigma_x(v) & \text{if $v\in V(P')$} \\
  \sigma_{x_1}(v) & \text{if $v\in  V(P'')$} \\
\end{cases}$
\end{center}
It can be easily shown  that $\sigma$ is a $(P_n,x)$-good path
2-placement. $\square$\\\\
Let $T$ be a non star tree and let $xy$ be an edge in $T$. We denote
by $T_{(x,y)}$ the connected component containing $x$ in $T-\{xy\}$
and it is called a neighbor tree of $y$. $T_{(x,y)}$ is said to be a
neighbor $F$-tree of $y$ if $T_{(x,y)}$ is a path of length at most
two such that whenever $T_{(x,y)}$
is a path of length two, then $x$ is an end vertex in it.\\

\begin{lemma}
Consider a non star tree $T$ containing a vertex x such that
$d(x)>2$. Let $\{x_1,...,x_n\}$, $n>2$, be the neighbors of $x$.
Suppose that
 $T_{(x_i,x)}$ is a neighbor $F$-tree of $x$ for $i=1,..., m$, where $2\leq m<n$. Let $T'$ be
the connected component containing $x$ in $T-\{xx_i;\;i=1,...,m\}$
and let $G$ be the graph obtained by the union of the remaining
components. Suppose that there exists a $(T',z)$-well 2-placement
$\sigma$ such that $dist (x,\sigma(x))\leq 3$, where $z$ is a vertex
in $T'$ not necessarily distinct from $x$, then there exists a
$(T,z)$-well 2-placement, say $\sigma_z$, such that
$\sigma_z(v)=\sigma(v)$ for every $ v\in V(T')$,
$dist(x_i,\sigma_z({x_i}))\leq 2$ for $i=1,...,m$ and
$dist_T(\sigma_z(u),\sigma_z(v))\leq 5$  whenever $uv$ is an edge in
$G$.
\end{lemma}
\begin{proof}
Let $r$, $p$ and $q$ be the number of neighbor trees of $x$ that are
paths of length zero, one and two, respectively, in the set
$\{T_{(x_i,x)};i=1,..., m\}$. In what follows we need to rename some
neighbors of $x$ for the sake of the proof. Let $T_i=T_{(x_i,x)}$
for $i=1,...,m$ such that if $r>0$, then $T_i$ is the vertex $a_i$
for $i=1,...,r$, $T_i=b_{i-r}c_{i-r}$ for $i=r+1,...,p+r$ if $p>0$
and $T_i=d_{i-(p+r)}e_{i-(p+r)}f_{i-(p+r)}$ for $i=r+p+1,...,r+p+q$
if $q>0$. We will define a $(T,z)$-well 2-placement $\sigma_z$
according to the different values of $r$, $p$ and $q$. To construct
$\sigma_z$, we need first to introduce the permutations  $\Theta$,
$\Upsilon$ and $\Delta$ on $V(G)$ in each case below
such that:\\
 $\Theta=\begin{cases}
\theta & \text{if r is even} \\
\theta' & \text{if r  is odd}\\
  \end{cases}$,
 $ \Upsilon=\begin{cases}
  \epsilon& \text{if p is even} \\
\epsilon' & \text{if p is odd}\\
  \end{cases}$ and
  $\Delta=\begin{cases}
  \delta & \text{if q is even} \\
\delta' & \text{if q is odd}\\
  \end{cases}$
  \\
  where $\theta$, $\theta'$, $\epsilon$, $\epsilon'$, $\delta$ and
  $\delta'$ are permutations defined in each case below.

\begin{itemize}
\item Case 1. $p$, $q$ and $r$ are strictly greater than one.\\
 If $r=2n'$ for some $n'\in
\mathbb{Z}$, then let $\theta=\prod_{j=1}^{j=n'}(a_{2j-1}\;a_{2j})$.
If $r=2n'+1$, then if $n'=1$ let $\theta'=(a_1\;a_2\;a_3)$, else let
$\theta'=(a_1\;a_2\;a_3)\prod_{j=2}^{j=n'}(a_{2j}\;a_{2j+1})$.
\\If $p=2m'$ for some $m'\in \mathbb{Z}$, then let
$\epsilon=\prod_{j=1}^{j=m'}(b_{2j}\;c_{2j}\;b_{2j-1}\;c_{2j-1})$.
If $p=2m'+1$, then let $\epsilon'=(b_1\;b_2\;b_3)(c_1\;c_3\;c_2)$ if
$m'=1$, else let
$\epsilon'=(b_1\;b_2\;b_3)(c_1\;c_3\;c_2)\prod_{j=2}^{j=m'}(b_{2j+1}\;c_{2j+1}\;b_{2j}\;c_{2j})$.\\If
$q=2s'$ for some $s'\in \mathbb{Z}$, then let
$\delta=\prod_{j=1}^{j=s'}(e_{2j-1}\;f_{2j}\;d_{2j-1})(e_{2j}\;f_{2j-1}\;d_{2j})$.
If $q=2s'+1$, then let
$\delta'=(d_1\;e_1\;f_2\;e_2\;f_1)(d_2\;d_3\;f_3\;e_3)$ if $s'=1$,
else let
$\delta'=(d_1\;e_1\;f_2\;e_2\;f_1)(d_2\;d_3\;f_3\;e_3)\prod_{j=2}^{j=s'}
(e_{2j}\;f_{2j+1}\;d_{2j})(e_{2j+1}\;f_{2j}\;d_{2j+1})$.\\
Finally, let  $\sigma_z=\Theta\; \Upsilon\; \Delta\; \sigma$.

\item Case 2. $r=1$.\\
We will study the following subcases:
\begin{enumerate}
\item $p>1$ and $q>1$.\\
In this case, if $p=2m'$ then  let
$\epsilon=(a_1\;b_1)(b_2\;c_2\;c_1)$ if $m'=1$, and if $m'>1$ then
let
 $\epsilon=(a_1\;b_1)(b_2\;c_2\;c_1)\prod_{j=2}^{j=m'} (b_{2j}\;c_{2j}\;b_{2j-1}\;c_{2j-1})$. On the other hand, if $p=2m'+1$, then let
$\epsilon'=(a_1\;b_1\;c_1)\prod_{j=1}^{j=m'}
(b_{2j+1}\;c_{2j+1}\;b_{2j}\;c_{2j})$. Finally, let
$\sigma_z=\Delta\;\Upsilon\;\sigma$, where $\delta$ and $\delta'$
are the same as in  Case 1.
\item $p>1$ and $q=1$.\\
Let $\sigma_z=\Upsilon \; (a_1\;d_1\;f_1\;e_1)\;\sigma$, where
$\epsilon$ and $\epsilon'$ are the same as in Case 1.
\item $p>1$ and $q=0$.\\
Then let $\sigma_z=\Upsilon\; \sigma$, where $\epsilon$ and
$\epsilon'$ are the same as in $(1)$.
\item $p=1$ and $q>1$.\\
Then $\sigma_z=(a_1\;b_1\;c_1)\; \Delta\;\sigma$,  where $\delta$
and $\delta'$ are the same as in  Case 1.
\item $p=1$ and $q=1$.\\
Then let $\sigma_z=(a_1\;b_1\;c_1\;e_1)(f_1\;d_1)\; \sigma$.
\item $p=1$ and $q=0$.\\
Then $\sigma_z=(a_1\;b_1\;c_1)\; \sigma$.
\item $p=0$ and $q>1$.\\
If $q=2s'$ then let $\delta=(d_1\;e_1\;f_2)(f_1\;a_1\;d_2\;e_2)$ if
$s'=1$ and if $s'>1$ then let
$\delta=(d_1\;e_1\;f_2)(f_1\;a_1\;d_2\;e_2)\prod_{j=2}^{j=s'}(e_{2j-1}\;f_{2j}\;d_{2j-1})(e_{2j}\;f_{2j-1}\;d_{2j})$.
On the other hand, if $q=2s'+1$ then let
$\delta'=(a_1\;d_1\;f_1\;e_1)\prod_{j=1}^{j=s'}(e_{2j}\;f_{2j+1}\;d_{2j})(e_{2j+1}\;f_{2j}\;d_{2j+1})$.
Finally, let $\sigma_z=\Delta\; \sigma$.
\item $p=0$ and $q=1$.\\
Then let $\sigma_z=(a_1\;d_1\;f_1\;e_1)\; \sigma$. \end{enumerate} 
\item Case 3. $p=1$.\\
\begin{enumerate}
\item $r>1$ and $q>1$.\\
If $r=2n'$ then let  $\theta=(a_1\;b_1\;c_1\;a_2)$ if $n'=1$ and if
$n'>1$ then let
$\theta=(a_1\;b_1\;c_1\;a_2)\prod_{j=2}^{j=n'}(a_{2j-1}\;a_{2j})$.
On the other hand, if $r=2n'+1$ then let
$\theta'=(b_1\;c_1\;a_1)\prod_{j=1}^{j=n'}(a_{2j}\;a_{2j+1})$. Let
$\sigma_z=\Theta\; \Delta\;\sigma$, where $\delta$ and $\delta'$ are
defined as in Case 1.
\item $r>1$ and $q=1$.\\
Then let $\sigma_z=(c_1\;e_1\;b_1)(f_1\;d_1)\Theta\;\sigma$, where
$\theta$ and $\theta'$ are the same as the ones defined in Case 1.
\item $r>1$ and $q=0$.\\
Then let $\sigma_z=\Theta\; \sigma$, where $\theta$ and $\theta'$
are just like the ones in $(1)$.
\item $r=0$ and $q>1$.\\
If $q=2s'$ then let
$\delta=(d_1\;f_1\;b_1\;c_1)(f_2\;d_2)(e_1\;e_2)$ if $s'=1$ and if
$s'>1$ then let
$\delta=(d_1\;f_1\;b_1\;c_1)(f_2\;d_2)(e_1\;e_2)\prod_{j=2}^{j=s'}(e_{2j-1}\;f_{2j}\;d_{2j-1})(e_{2j}\;f_{2j-1}\;d_{2j})
$. Otherwise, if $q=2s'+1$ then
$\delta'=(c_1\;e_1\;f_1\;d_1\;b_1)\prod_{j=1}^{j=s'}(e_{2j}\;f_{2j+1}\;d_{2j})(e_{2j+1}\;f_{2j}\;d_{2j+1})$.
Finally, let $\sigma_z=\Delta\;\sigma$.
\item $r=0$ and $q=1$.\\
Let $\sigma_z=(c_1\;e_1\;f_1\;d_1\;b_1)\;\sigma$.
\end{enumerate}
\item Case 4. $q=1$.\\
\begin{enumerate}
\item $r>1$ and $p>1$.\\
If $r=2n'$ then let $\theta=(d_1\;f_1\;e_1\;a_1\;a_2)$ if $n'=1$ and
if $n'>1$ then let $\theta=
(d_1\;f_1\;e_1\;a_1\;a_2)\prod_{j=2}^{j=n'}(a_{2j-1}\;a_{2j})$. If
$r=2n'+1$, then let
$\theta'=(a_1\;d_1\;f_1\;e_1)\prod_{j=1}^{j=n'}(a_{2j}\;a_{2j+1})$.
Finally, let $\sigma_z=\Theta \; \Upsilon\;\sigma$, where $\epsilon$
and $\epsilon'$ are the same as the ones defined in Case 1.
\item $r>1$ and $p=0$.\\
Then let $\sigma_z=\Theta\;\sigma$, where $\theta$ and $\theta'$ are
the same as in $(1)$.
\item $r=0$ and $p>1$.\\
If $p=2m'$ then if $m'=1$ let
$\epsilon=(e_1\;b_2\;\;b_1\;c_1\;c_2)(d_1\;f_1)$
 and if $m'>1$ then let
$\epsilon=(e_1\;b_2\;\;b_1\;c_1\;c_2)(d_1\;f_1)\prod_{j=2}^{j=m'}(b_{2j}\;c_{2j}\;b_{2j-1}\;c_{2j-1})$.
On the other hand, if $q=2m'+1$, then let
$\epsilon'=(e_1\;f_1\;d_1\;b_1\;c_1)\prod_{j=1}^{j=m'}
(b_{2j}\;c_{2j}\;b_{2j+1}\;c_{2j+1})$. Finally, let
$\sigma_z=\Upsilon\;\sigma$.
\end{enumerate}
\item Case 5. $r=0$.\\
 \begin{enumerate}
\item $p>1$ and $q>1$.\\
Let $\sigma_z=\Upsilon\;\Delta\;\sigma$, where $\delta,\;\delta',\;
\epsilon$ and $\epsilon'$ are the same as the ones defined in  Case
1.
\item $p>1$ and $q=0$.\\
Then let $\sigma_z=\Upsilon\;\sigma$, where $ \epsilon$ and
$\epsilon'$ are the same as the ones defined in  Case 1.
\item $p=0$ and $q>1$.\\
Then let $\sigma_z=\Delta\;\sigma$, where $\delta$ and $\delta'$ are
the same as the ones defined in  Case 1.
\end{enumerate}
\item Case 6. $p=0$.\\
\begin{enumerate}
\item  $r>1$ and $q>1$.\\
Then let $\sigma_z=\Theta\;\Delta\;\sigma$, where
$\delta,\;\delta',\; \theta$ and $\theta'$ are the same as the ones
defined in  Case 1.
\item $r>1$ and $q=0$.\\
Then let $\sigma_z=\Theta\;\sigma$, where $\theta$ and $\theta'$ are
the same as
the ones defined in  Case 1.\\
\end{enumerate}
\item Case 7. $q=0$.\\
We still have only the case where $r>1$ and $p>1$. Then let
$\sigma_z=\Theta\;\Upsilon\;\sigma$, where $\theta,\;\theta',\;
\epsilon$ and $\epsilon'$ are the same as the ones defined in Case
1.
\end{itemize}
Note that in each of the above cases, $dist_T(x_i,\sigma_z(x_i))\leq
2$ for $i=1,...,m$, $dist_T(c_l,\sigma_z(c_l))\leq 4$ for
$l=1,...,p$, $dist_T(f_j,\sigma_z(f_j))\leq 4$ for $j=1,...,q$ and
$dist_T(\sigma_z(u),\sigma_z(v))\leq 5$ whenever $uv$ is an edge in
$G$. Thus, it can be easily proved that $\sigma_z$ is a $(T,z)$-well
2-placement.
\end{proof}
\begin{corollary}
Consider a non star tree containing a vertex x such that $d(x)>2$.
Let $\{x_1,...,x_n\}$, $n>2$, be the neighbors of $x$. Suppose that
$T_{(x_i,x)}$ is a neighbor $F$-tree of $x$ for $i=1,..., m$, where
$2\leq m< n$. Let $T'$ be the connected component containing $x$ in
$T-\{xx_i;\;i=1,...,m\}$. Suppose that there exists a $(T',z)$-good
2-placement $\sigma$ such that $dist (x,\sigma(x))\leq 2$, where $z$
is a vertex in $T'$ not necessarily distinct from $x$, then there
exists a $(T,z)$-good 2-placement say $\sigma_z$ such that
$\sigma_z(v)=\sigma(v)$ for every $ v\in V(T')$ and
$dist(x_i,\sigma_z(x_i))\leq 2$ for $i=1,...,m$.
\end{corollary}
\begin{proof}
Let $G$ be the graph obtained by the union of all the components
that do not contain  $x$ in $T-\{xx_i;\;i=1,...,m\}$. Since $\sigma$
is a $(T',z)$-good 2-placement then it is a $(T',z)$-well
2-placement, and so, by Lemma 3.3, there exists a $(T,z)$-well
2-placement, say $\sigma_z$, such that $\sigma_z(v)=\sigma(v)$ for
every $v\in V(T')$, $dist(x_i,\sigma_z({x_i}))\leq 2$ for
$i=1,...,m$ and $dist(\sigma_z(u),\sigma_z(v))_T\leq 5$ whenever
$uv$ is an edge in $G$. Thus $\sigma_z$ is a $(T,z)$-good
2-placement.
\end{proof}
\begin{lemma}
Let $T$ be a non star tree and let $x$ be a vertex of $T$ with
$N(x)=\{x_1,...,x_n\}$, $n\geq 2$. Suppose that there exists a
$(T_{(x_i,x)}, x_i)$-well 2-placement  for $i=1,..., p$, where
$1\leq p< n$. Let $T'$ be the connected component containing $x$ in
$T-\{xx_i;\; i=1,...,p\}$. If there exists a $(T',z)$-well
2-placement, say $\sigma$, such that $dist(x,\sigma(x))\leq 3$,
where $z$ is a vertex of $T'$,  then there exists a $(T,z)$-well
2-placement, say $\sigma_z$, such that $\sigma_z(v)=\sigma(v)$ for
every $ v \in V(T')$.\end{lemma}
\begin{proof}
Let $\sigma_i$ be a  $(T_{(x_i,x)},x_i)$-well 2-placement for
$i=1,...,p$. Then the 2-placement $\sigma$ defined as follows:
\begin{center}
$\sigma_z(v)=\begin{cases}
 \sigma(v) &\text{if $v\in V(T')$}\\
 \sigma_i(v) & \text {if $v\in V(T_{x_ix})$ for $i=1,..., p$}\\
 \end{cases}$
 \end{center} is a $(T,z)$-well 2-placement.
 \end{proof}
\begin{lemma}
Let $T$ be a non star tree and let $x$ be a vertex of $T$ with
$N(x)=\{x_1,...,x_n\},$ $n\geq 4$. Let $T'$ be the connected
component containing $x$ in $T-\{xx_i;\; i=1,..., p\},\; 3\leq p<n$.
Suppose that at least two and at most $p-1$ trees  in the set
$\{T_{(x_i,x)}:\;i=1,...,p\}$ are neighbor $F$-trees of $x$ such
that for the remaining non neighbor $F$-trees in this set there
exists a $(T_{(x_i,x)},x_i)$-well 2-placement.  If there exists a
$(T',z)$-well 2-placement, say $\sigma'$, such that
$dist(x,\sigma'(x))\leq 3$, where $z$ is a vertex of $T'$ not
necessarily distinct from $x$, then there exists a $(T,z)$-well
2-placement, say $\sigma$, such that $\sigma(v)=\sigma'(v)$ for
every $ v \in V(T')$.
\end{lemma}
\begin{proof}
Suppose that $T_{(x_i,x)}$ are the  neighbor $F$-trees of $x$ for
$i=1,...,k$, where $2\leq k\leq p-1$, such that there exists a
$(T_{(x_i,x)},x_i)$-well 2-placement for $i=k+1,...,p$. Let $T''$ be
the connected component containing $x$ in $T-\{xx_i;\;
i=k+1,...,p\}$. Since $dist(x,\sigma'(x))\leq 3$  and since
$T_{(x_i,x)}$ are neighbor $F$-trees of $x$ for $i=1,...,k$,
 then, by Lemma 3.3, there exists a $(T'',z)$-well 2-placement, say
$\sigma''$, such that $\sigma'(v)=\sigma''(v)$ for every $ v \in
V(T')$. Again, since $dist(x,\sigma''(x))\leq 3$ and there exists a
$(T_{(x_i,x)},x_i)$-well 2-placement for $i=k+1,...,p$, then, by the
previous Lemma, there exists a $(T,z)$-well 2-placement, say
$\sigma$, such that $\sigma(v)=\sigma''(v)$ for every $ v \in
V(T'')$.
 \end{proof}
\begin{figure}
   \centering
 \fbox{\includegraphics[width=0.7\textwidth]%
    {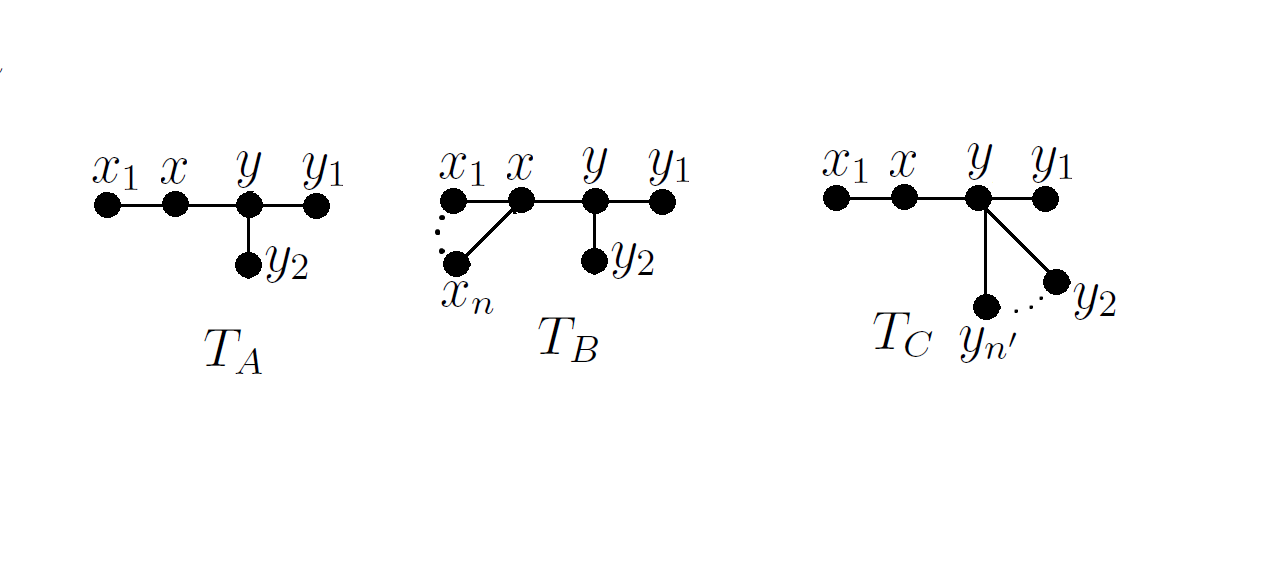}}
    \center{\figurename{\;1}}
\end{figure}

\begin{lemma}
Let $T$ be one of the  trees in Fig. 1 such that $n\geq 2$ and
$n'\geq 3$. Then there exists  a $(T,x_1)$-good 2-placement and a
$(T, v)$-well 2-placement for every vertex $v$ of $T$.
\end{lemma}
\begin{proof}
For each tree $T$  in Fig. 1, we give below a $(T,x_1)$-good
2-placement and a $(T, v)$-well 2-placement for every vertex $v$ of
$T$.
\\
$\sigma=(x_1\;x\;y_1\;y_2\;y)$ is  a $(T_A,x_1)$-good 2-placement
 and a $(T_A,v)$-well 2-placement
for every $ v \in V({T_A})$.\\
There are four choices for choosing a vertex $v$ of $T_B$, either
$x_1$, $x$, $y$ or $y_1$. If $n=2$, then
$\sigma=(x_2\;y_2)(y_1\;y\;x_1\;x)$ is a $(T_B,x_1)$-good
2-placement
 and a $(T_B,v)$-well 2-placement, for every $ v \in \{x_1,x,y,y_1\}$.
Otherwise, suppose that $n\geq 3$. If $n=2k+1$ for some $ k\in
\mathbb{N}$, then
$\sigma=(x_1\;x\;y_1\;y_2\;y)\prod_{i=1}^{i=k}(x_{2i}\;x_{2i+1})$ is
a $(T_B,x_1)$-good 2-placement
 and a $(T_B,v)$-well 2-placement $\forall\; v \in V({T_B})$. On
the other hand, if $n=2k$ for some $k\in \mathbb{N}$, then
$\sigma=(x_2\;y_2)(y_1\;y\;x_1\;x)\prod_{i=2}^{i=k}(x_{2i-1}\;x_{2i})$
is a $(T_B,x_1)$-good 2-placement
 and a $(T_B,v)$-well 2-placement, for every $ v \in \{x_1,x,y,y_1\}$.\\
 Finally, if $n'=2k+1$ for some $k\in \mathbb{N}$, then
$\sigma=(x_1\;x\;y_1\;y)\prod_{i=1}^{i=k}(y_{2i}\;y_{2i+1})$ is a
$(T_C,x_1)$-good 2-placement
 and a $(T_C,v)$-well 2-placement
$\forall\; v \in V({T_C})$.\\
And if $n'=2k$ for some $k\in \mathbb{N},\; k\geq 2$, then
$\sigma=(x_1\;x\;y_1\;y_2\;y)\prod_{i=2}^{i=k}(y_{2i-1}\;y_{2i})$
 is  a $(T_C,x_1)$-good 2-placement
 and a $(T_C,v)$-well 2-placement
 for every $ v \in V({T_C})$.
\end{proof}
\begin{figure}[h]
   \centering
 \fbox{\includegraphics[width=0.7\textwidth]%
    {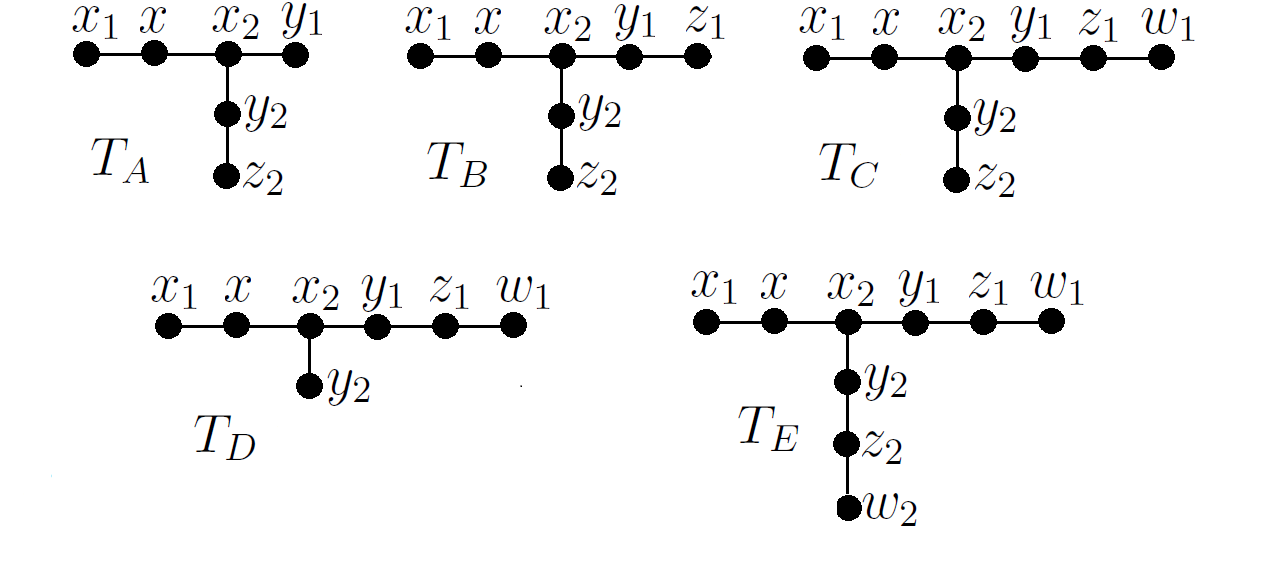}}
    \center{\figurename{\;2}}
\end{figure}
\begin{lemma} Let $T$ be one of the  trees in Fig. 2.  Then there exists a
$(T, x)$-well 2-placement and a $(T,x_1)$-well 2-placement.
\end{lemma}
\begin{proof}
For each pair $(T, x)$ and $(T,x_1)$ in Fig. 2, we give below a $(T,
x)$-well 2-placement $\sigma_x$ and a $(T,x_1)$-well 2-placement $\sigma_{x_1}$.\\
For $T_A$, $\sigma_{x}=\sigma_{x_1}=(x_1\;x\;y_2)(x_2\;y_1\;z_2)$.\\
For $T_B$, $\sigma_{x}=\sigma_{x_1}=(x_1\;x\;y_2)(y_1\;z_2)(x_2\;z_1)$.\\
For $T_C$, $\sigma_{x}=\sigma_{x_1}=(x_1\;x\;y_2)(y_1\;z_2)(x_2\;w_1\;z_1)$.\\
For $T_D$, $\sigma_{x}=\sigma_{x_1}=(x_1\;x\;y_2)(x_2\;y_1\;w_1\;z_1)$.\\
For $T_E$,
$\sigma_{x}=\sigma_{x_1}=(x_1\;x\;y_2)(z_2\;w_2\;y_1)(z_1\;x_2\;w_1)$.
\end{proof}

\begin{lemma}
Consider a non star tree $T$ containing an edge  $xy$ such that
$d_T(x)=1$ and $d_T(y)\geq 3$. Suppose that all the neighbor trees
of $y$ distinct from $T_{(x,y)}$ are neighbor $F$-trees not
isomorphic to $P_1$, then there exists a $(T,x)$-well 2-placement.
\end{lemma}
\begin{proof}
Let $\{y_1,...,y_n\}$, where $n\geq 2$, be the neighbors of $y$
distinct from $x$. We need to consider the following two cases
concerning the degree of $y$:
\begin{enumerate}
\item $d(y)>3$.\\
Let $T_0$ be the connected component containing $x$ in $T-\{yy_i;
i=2,..., n\}$, then $T_0$ is  isomorphic to $P_4$ or $P_5$, and so,
by lemma 3.1, there exists a $(T_0,x)$-well path 2-placement, say
$\sigma_0$, such that $dist(v,\sigma_0(v))\leq 3$ for every vertex
$v$ of $T_0$.
Since $T_{(y_i,y)}$ are neighbor $F$-trees of $y$ for $i=2,..., n$, then,  by Lemma 3.3, there exists a $(T,x)$-well 2-placement.\\

\item $d(y)=3$\\
Then $T$ is isomorphic to one of the following trees in Fig. 3.
\begin{figure}[h]
   \centering
 \fbox{\includegraphics[width=0.7\textwidth]%
    {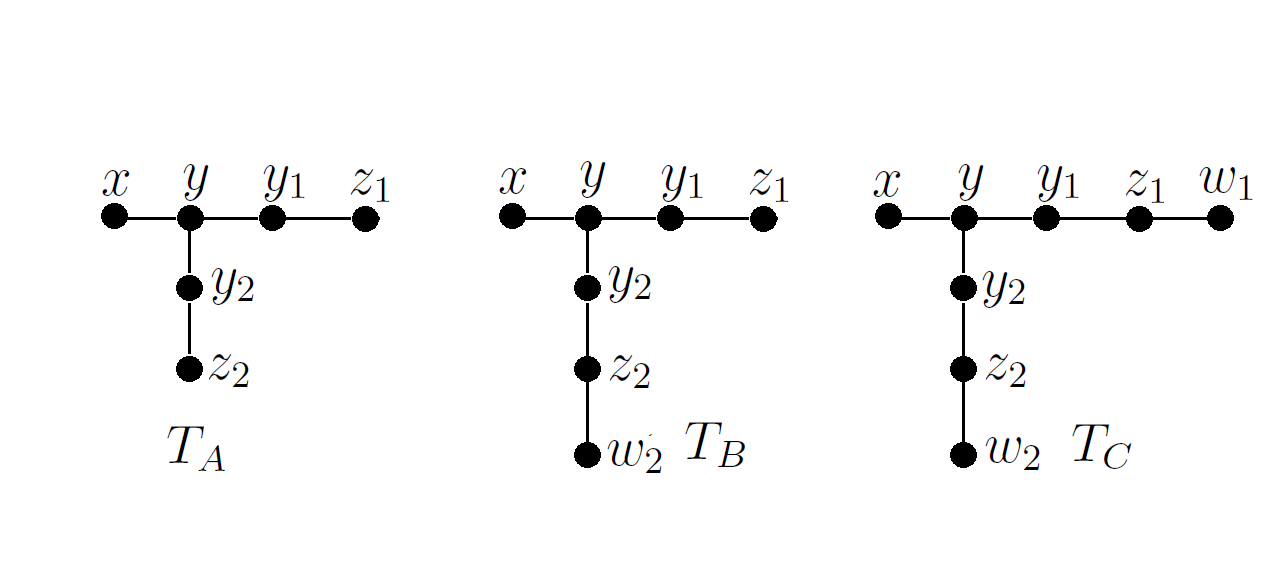}}
    \center{\figurename{\;3}}
\end{figure}
For each pair $(T, x)$ in Fig. 3, we give below a $(T, x)$-well
2-placement $\sigma$:\\
$\sigma=(x\;y\;z_2)(y_1\;z_1\;y_2)$ is a $(T_A,x)$-well 2-placement.\\
$\sigma=(x\;y\;z_1)(y_1\;z_2)(y_2\;w_2)$ is a $(T_B,x)$-well 2-placement.\\
$\sigma=(x\;y\;z_1)(y_1\;z_2\;w_1)(w_2\;y_2)$ is a $(T_C,x)$-well
2-placement.
\end{enumerate}
\end{proof}
\begin{figure}[h]
   \centering
 \fbox{\includegraphics[width=0.7\textwidth]%
    {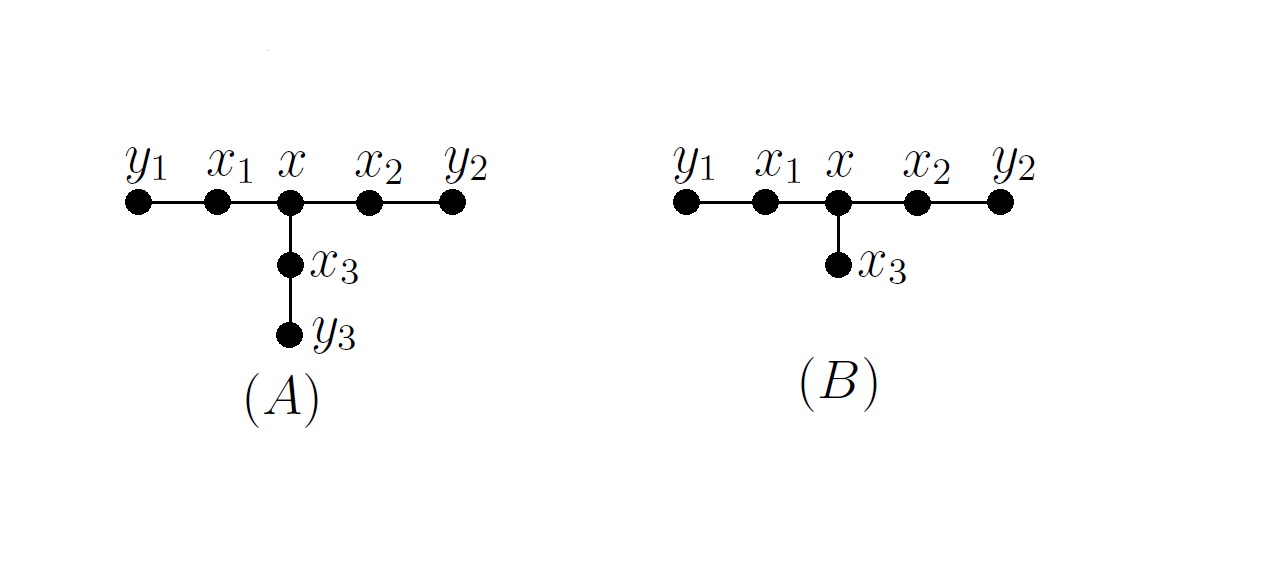}}
    \center{\figurename{\;4}}
\end{figure}
\begin{lemma}
Consider a non star tree containing a vertex $x$ with $d(x)\geq 3$.
Suppose that all the neighbor trees of $x$ are isomorphic either to
$P_1$ or $P_2$ such that $x$ has at most one neighbor tree
isomorphic to $P_1$, then there exists a $(T,x)$-well 2-placement.
\end{lemma}
\begin{proof}
Let $N(x)=\{x_1,...,x_n\},$  $ n\geq 3$. If $d(x)>3$, then let $T_0$
be the connected component containing $x$ in $T-\{xx_i;\;i=3,...,
n\}$, then $T_0$ is isomorphic either to $P_4$ or $P_5$. Hence, by
Lemma 3.1, there exists a $(T_0,x)$-well path 2-placement, say
$\sigma_0$. Since $dist(x,\sigma_0(x))\leq 2$ and $T_{(x_i,x)}$ are
neighbor $F$-trees for $i=3,...,n$, then, by Lemma 3.3 ,there exists
a $(T,x)$-well 2-placement. On the other hand, if $d(x)=3$, then $T$
is isomorphic  to one of the trees in Fig. 4. For each
pair $(T, x)$ in Fig. 4, we give below a $(T, x)$-well 2-placement $\sigma$:\\
 For $T_A$, $\sigma=(x\;y_3)(y_2\;x_2\;x_1)(x_3\;y_1)$\\
For $T_B$, $\sigma=(x\;y_1)(x_1\;y_2)(x_3\;x_2)$\\
 \end{proof}
\begin{lemma}
 Consider a non star tree $T$ containing an edge  $xy$ such that
$d_T(x)=1$ and $d_T(y)\geq 3$. Suppose that all the neighbor trees
of $y$ other than $T_{(x,y)}$ are neighbor $F$-trees not isomorphic
to $P_1$.
 Then there exists a $(T,x)$-good 2-placement.
 \end{lemma}
 \begin{proof}
 Let $\{y_1,...,y_n\}$ be the set of neighbors
 of $y$ distinct from $x$, where $n\geq 2$. We are going to study two cases:
\begin{itemize}
\item $y$ has a neighbor tree isomorphic to $P_2$.\\
Without loss of generality, suppose that $T_{(y_1,y)}=y_1z_1$. If
$n>2$, then let $T_0=T-\{yy_i:\; i=2,....,n\}$
$\sigma_0=(x\;y\;z_1\;y_1)$ be a $(T_0,x)$-good 2-placement. Since
$dist (y,\sigma_0(y))=2$, then, by Corollary 3.1, there exists a
$(T,x)$-good 2-placement. Otherwise, that is, $n=2$, then $T$ is
either isomorphic to $T_A$ or $T_B$
in Fig. 3. For each case, we give below a $(T, x)$-good 2-placement $\sigma$:\\
 For $T_A$, $\sigma=(x\;y\;z_1)(y_1\;y_2\;z_2)$. \\
For $T_B$, $\sigma=(x\;y\;z_1\;y_1\;z_2)(y_2\;w_2)$.
\item $y$ has a neighbor tree isomorphic to $P_3$.\\
Without loss of generality, suppose that $T_{(y_1,y)}=y_1z_1w_1$. If
$n>2$, then let let $T_0=T-\{yy_i:\; i=2,....,n\}$ and
$\sigma_0=(x\;y\;z_1\;w_1\;y_1)$ be a $(T_0,x)$-good 2-placement.
Since $dist (y,\sigma_0(y))=2$, then, by Corollary 3.1, there exists
a $(T,x)$-good 2-placement. Otherwise, that is, $n=2$, then $T$ is
isomorphic either to $T_B$ or $T_C$ in Fig. 3. We showed above that
there exists a $(T_B,x)$-good 2-placement. For $T_C$,
$\sigma=(x\;y\;z_1\;w_1\;y_1\;y_2\;w_2\;z_2)$ is a
$(T_C,x)$-good 2-placement.\end{itemize}\end{proof} 
\begin{lemma}
Let $T$ be one of the  trees in Fig. 5 such that $n\geq 2$.  Then
there exists a $(T,x)$-good 2-placement $\sigma_x$ such that
$dist(a_1,\sigma_x(a_1))\leq 2$ whenever $T$ is isomorphic to $T_E$,
$T_F$ or $T_G$.
\end{lemma}
\begin{proof}
\begin{figure}[h!]
   \centering
 \fbox{\includegraphics[width=0.7\textwidth]%
    {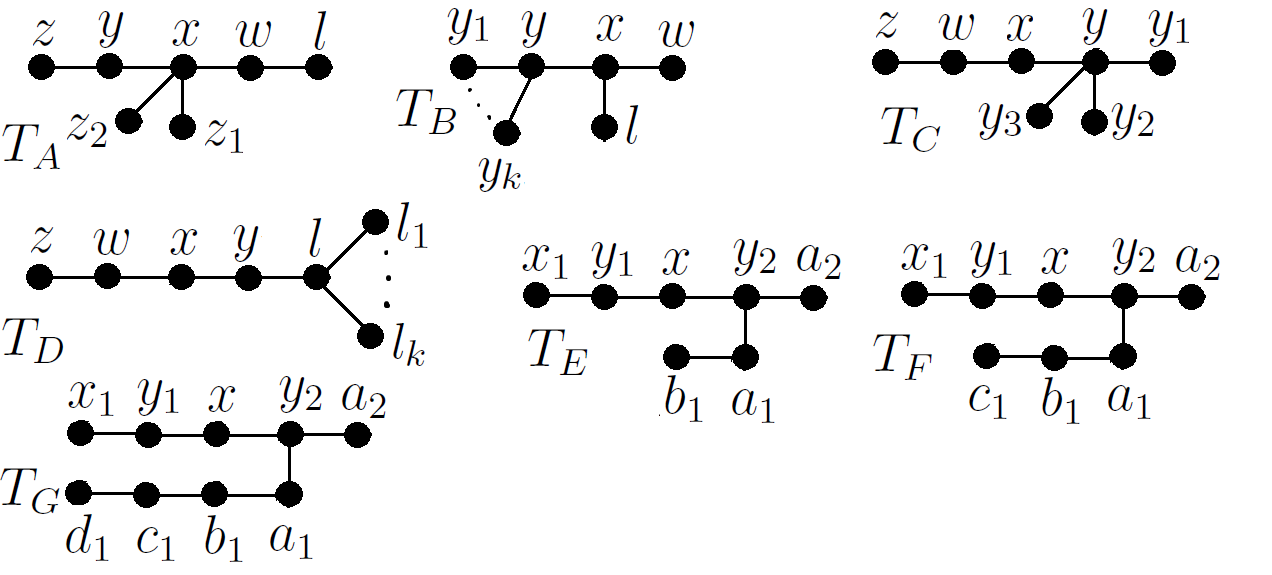}}
    \center{\figurename{\;5}}
\end{figure}
For each pair $(T, x)$ in Fig. 5, we give below a $(T,x)$-good
2-placement $\sigma_x$.
\\
For $T_A$, $\sigma_{x}=(x\;z_1\;z_2\;w\;l\;y\;z)$\\
For $T_B$, if $k=2p$ for some $p\in \mathbb{N}$, then
$\sigma_{x}=(x\;w\;l\;y\;y_1\;y_2)$   if $k=2$, and if $k>2$, then
$\sigma_{x}=(x\;w\;l\;y\;y_1\;y_2)\prod_{i=2}^{i=p}
(y_{2i-1}\;y_{2i})$. If $k=2p+1$, then
$\sigma_{x}=(x\;w\;l\;y\;y_1)\prod_{i=1}^{i=p}
(y_{2i+1}\;y_{2i})$ \\
For $T_C$, $\sigma_{x}=(x\;y\;w\;z)(y_1\;y_2\;y_3)$.\\
For $T_D$, if $k=2p$ for some $p\in \mathbb{N}$, then
$\sigma_{x}=(x\;y\;l_1\;l_2\;l\;w\;z)$ if $k=2$, and if $k>2$,
$\sigma_{x}=(x\;y\;l_1\;l_2\;l\;w\;z)\prod_{i=2}^{i=p}
(l_{2i-1}\;l_{2i})$. If $n=2p+1$ , then
$\sigma_{x}=(x\;y\;l_1\;l\;w\;z)\prod_{i=1}^{i=p}
(l_{2i+1}\;l_{2i})$.\\
For $T_E$, $\sigma_x=(x\; y_2\; y_1\; x_1)(a_1\;b_1\;a_2)$.\\
For $T_F$, $\sigma_x=(x\; y_2\; y_1\; x_1)(a_1\;c_1\;b_1\;a_2)$.\\
For $T_G$, $\sigma_x=(x\; y_2\; y_1\; x_1)(a_1\;c_1\;d_1)(b_1\;a_2)$.\\
\end{proof}

\begin{figure}[h!]
   \centering
 \fbox{\includegraphics[width=0.5\textwidth]%
    {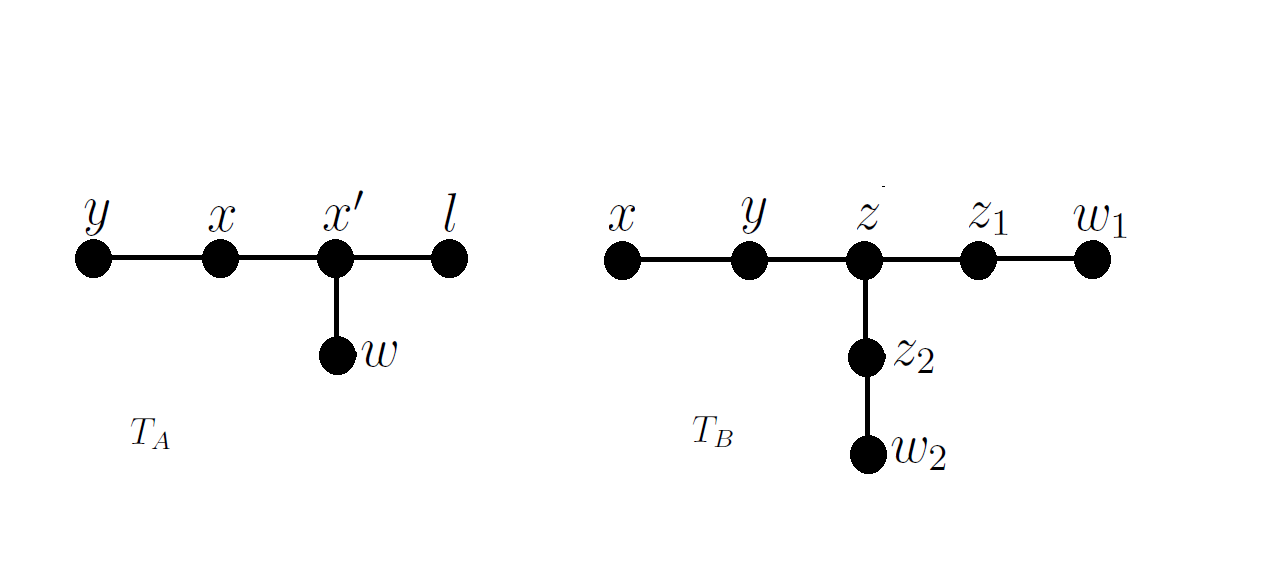}}
    \center{\figurename{\;6}}
\end{figure}

\begin{lemma}
For the trees shown in Fig. 6, there exists a $(T_A,x)$-good
2-placement, a $(T_A,x')$-good 2-placement and a $(T_B,x)$-good
2-placement.
\end{lemma}
\begin{proof}
$\sigma=(x\; y\; x'\; w\;l)$ is a  $(T_A,x)$-good 2-placement and a
$(T_A,x')$-good 2-placement.\\
$\sigma=(x\;y\;z_1\;w_1\;z_2)(w_2\;z)$ is a $(T_B,x)$-good
2-placement.
\end{proof}\\\\
\textbf{Proof of Theorem 1.7}.\\
 The proof is by induction on the
order of $T$. For $n=4$, there is only one non star tree, which is
$P_4$, then, by Lemma 3.1, there exists  a $(T,v)$-well 2-placement
for every $ v \in V(P_4)$. Suppose that the theorem holds for
$n'<n$, $n\geq 5$, and let $T$ be a non star tree of order $n$. Let
$x$ be a vertex of $T$. If $T$ is a path, then, by Theorem 3.1,
there exists a $(T,x)$-well path 2-placement. Otherwise, let $v$ be
a vertex of $T$ and let $\{v_1,...,v_m\}$, $m\geq 2$, be the
neighbors of $v$. Suppose that $\{v_1,...,v_p\}$, $1\leq p< m$, are
enumerated in such a way that $x$ and $v$ are in the same connected
component in $T-\{vv_i:i=1,...,p\}$ and let $T'$ be this connected
component.
Note that $x$ and $v$ may be the same vertex. \\\\
\textbf{Claim 1.} If there exists a $(T',x)$-well 2-placement
$\sigma$ such that $dist(v,\sigma(v))\leq 3$ and $T_{(v_i,v)}$ are
non star trees
for $i=1,...,p$, then there exists a $(T,x)$-well 2-placement.\\\\
\textbf{Proof.} Since $T_{(v_i,v)}$ is a non star tree for
$i=1,...p$, then, by induction, there exists a
$(T_{(v_i,v)},v_i)$-well 2-placement. Thus,  by Lemma 3.4,
there exists a $(T,x)$-well 2-placement.\\\\
\textbf{Claim 2.} If $2\leq p< m$, at least two neighbor trees in
the set $\{T_{(v_i,v)}:\; i=1,...,p\}$ are neighbor $F$-trees of $v$
 such that the remaining neighbor trees in the set are non star trees and if there  exists a $(T',x)$-well 2-placement $\sigma$ such that
$dist(v,\sigma(v))\leq 3$,
then there exists a $(T,x)$-well 2-placement.\\\\
\textbf{Proof.} If $T_{(v_i,v)}$ is a non star tree for some $i$,
$1\leq i\leq p$, then, by induction, there exists a
$(T_{(v_i,v)},v_i)$-well 2-placement. Thus, since at least two
neighbor trees in the set $\{T_{(v_i,v)}:\; i=1,...,p\}$, are
neighbor $F$-trees of
$v$,  there exists, by Lemma 3.3 and Lemma 3.5, a $(T,x)$-well 2-placement.\\\\
From now on, we shall assume that we can't apply neither Claim 1 nor
Claim 2 on any vertex $v$ of $T$. We will study two cases concerning
the degree
of $x$.\\\\
\textbf{Case 1.} $d(x)=1$.\\
Let $y$ be the father of $x$. If $y$ is the father of another leaf,
say $\alpha$, then let $T'=T-\{x\}$. Note that $T'$ is a non star
tree since otherwise $T$ is a star. So, there exists a
$(T',\alpha)$-well 2-placement, say $\sigma'$. The 2-placement
$\sigma$ defined as follows:
\begin{center}
$\sigma(v)=\begin{cases}
 \sigma'(v) &\text{if $v\in V(T)-\{\alpha, x\}$}\\
 x & \text {if $v=\alpha$}\\
 \sigma'(\alpha)& \text{if $v=x$}\\
 \end{cases}$
 \end{center} is a $(T,x)$-well 2-placement.
 Otherwise, suppose that $y$ is the father of the leaf $x$ only. If
 there exists a set of leaves, say $\{\alpha_1,...,\alpha_k\}$, $\;k\geq 2$,
 such that all of these leaves have the same father, say $\beta$, with $d(\beta)=k+1$, then
 let $T'=T- \{\alpha_1,...,\alpha_k\}$. If $T'$ is a non
 star tree then there exists a $(T',x)$-well 2-placement, say
 $\sigma_x$, such that $dist(\beta,\sigma_x(\beta))\leq 4$, since
 $\beta$ is a leaf in $T'$, and so $\sigma=\sigma_x(\alpha_1...\;\alpha_k)$ is a
 $(T,x)$-well 2-placement. Else, that is  $T'$ is a star, then $T$ is isomorphic to $T_A$ in Fig. 1 if
 $k=2$,
 and if $k>2$ then $T$  is isomorphic to $T_C$ in Fig. 1 for $n'=k$. Thus, by Lemma 3.6, there
 exists a $(T,x)$-well 2-placement. Otherwise, suppose that the set
 $\{\alpha_1,...,\alpha_k\}$ doesn't exist. Hence, we can remark that for any edge $ab$ in $E(T)$, $T_{(a,b)}$
 is either a neighbor $F$-tree of $b$ or a non star tree. Since $T$ is
 not a path, then there exists a vertex in $T$ with degree strictly greater than two. Let $z$ be the first
 vertex away from $x$ such that $d(z)>2$, $z'$ be the nearest neighbor of
$z$ to $x$ and let $\{z_1,...,z_m\}$, $m\geq 2$, be the neighbors of
$z$ distinct from $z'$. Note that $T_{(z',z)}$ is a path and
$T_{(z,z')}$ is a non star tree since otherwise the set
$\{\alpha_1,...,\alpha_k\}$, which is described above,  exists.
$T_{(z',z)}$ is a path of length at most two, since otherwise
 there exists, by Theorem 3.1, a $(T_{(z',z)},x)$-well path 2-placement, say $\sigma_x$, such that $dist(z',\sigma_x(z'))\leq
3$, and so we can apply Claim 1 on $z'$. Suppose first that the path
$T_{(z',z)}$ is of length zero,
 that is $x$ is a neighbor of $z$, then  $z$ is the father of the leaf $x$ only.
  If $z$ has a non neighbor $F$-tree, then suppose that $T_{(z_1,z)}$ is a non star tree.
 Since $T_{(z,z_1)}$ is a non star tree, then there exists a $(T_{(z,z_1)},x)$-well 2-placement $\sigma'$ such that
 $dist(z,\sigma'(z))\leq 3$, and so we can apply Claim 1 on $z$, a contradiction. Hence, all the neighbor trees of $z$ are
 neighbor $F$-trees and so there exists a $(T,x)$-well 2-placement by
 Lemma 3.8. Suppose now that the path $T_{(z',z)}$ is  of length two or one.
 If $z$ has a unique neighbor $F$-tree or
 at least three neighbor $F$-trees other that $T_{(z',z)}$, then
 suppose that $T_{(z_1,z)}$ is a neighbor $F$-tree and let $T'$ be
 the connected component containing $x$ in $T-\{zz_i;\;
 i=2,...,m\}$. Thus, $T'$ is a path of length at least three and at
 most six and so, by Lemma 3.1, there exists a $(T',x)$-well path
 2-placement $\sigma'$ such that $dist(v,\sigma'(v))\leq 3$ for every vertex $v$ of $T'$. Hence,
 we can apply Claim 1 or Claim 2 on $z$, a contradiction. Thus,
  $z$ has no neighbor $F$-tree  or $z$ has only two neighbor $F$-trees distinct from $T_{(z',z)}$.
 If $T_{(z',z)}$ is a path of length two, then let $T'$ be the connected component
 containing $x$ in $T-\{zz_i;\; i=1,...,m\}$. Since $T'$ is a path of length three, then there exists a
 $(T',x)$-well path 2-placement $\sigma'$, by Lemma 3.1, such that
 $dist(z,\sigma'(z))\leq 3$, and so we can apply Claim 1 or Claim 2 on $z$, a
 contradiction. Thus,
 $T_{(z',z)}$ is a path of length one, that is $z'$ is the father of
 $x$. If $z$ has no neighbor $F$-tree distinct from $T_{(z',z)}$, then
  each neighbor of $z$ has at least two neighbors. Let $\{a_1,...,a_q\}$, $q\geq 1$, be the neighbors of $z_1$ distinct from
 $z$. If $z_1$ has at least two neighbor $F$-trees or no neighbor $F$-tree distinct from $T_{(z,z_1)}$, then
 let $T'$ be the connected component
 containing $x$ in $T-\{zz_j,\;z_1a_i;\;j=2,..., m\; and \; i=1,...,q\}$. Thus, $T'$ is
 a path of length three and so, by Lemma 3.1, there exists a
 $(T',x)$-well path 2-placement $\sigma'$ such that
 $dist(v,\sigma'(v))\leq 3$ for every vertex $v$ of $T'$. Since $T_{(z_j,z)}$ are non neighbor $F$-trees for $j=2,...,m$, then, by Lemma
 3.4, there exists a $(T'',x)$-well 2-placement $\sigma''$,
 where $T''$ is the connected component containing $x$ in
 $T-\{z_1a_i;\; \;i=1,...,q \}$, such that
 $\sigma''(v)=\sigma'(v)$ for every vertex $v$ of $T'$  and so we can either apply Claim 1 or Claim 2 on $z_1$, a contradiction.
Thus, $z_1$ has a unique neighbor
 $F$-tree distinct from $T_{(z,z_1)}$, say $T_{(a_1,z_1)}$. Let $T'$ be the connected component containing $x$ in
 $T-\{zz_j,\;z_1a_i;\; j=2,...,m\; and\;i=2,...,q \}$, then $T'$ is a path of length
 at least  four and at most six, and so, by Lemma 3.1, there exists
 a $(T',x)$-well path 2-placement  $\sigma'$ such that
 $dist(v,\sigma'(v))\leq 3$ for every vertex $v$ of $T'$. Thus, by Lemma
 3.4, there exists a $(T'',x)$-well  2-placement $\sigma''$,
 where $T''$ is the connected component containing $x$ in
 $T-\{z_1a_i;\; \;i=2,...,q \}$, such that
 $\sigma''(v)=\sigma'(v)$ for every $v$ of $T'$ and so we can either apply Claim 1 on $z_1$, a contradiction.
 Thus, $z$ has only two neighbor $F$-trees distinct from
 $T_{(z',z)}$. If $m>2$, then suppose that $T_{(z_1,z)}$ and
 $T_{(z_2,z)}$ are the  neighbor $F$-trees of $z$. The tree $T'$, which is the connected component containing $x$ in $T-\{zz_i;\;
 i=3,...,m\}$,
 is either isomorphic to one of the trees in Fig. 2 or to the tree $T_A$ in Fig. 1, and so, by Lemma 3.6 and Lemma 3.7,
 there exists a $(T',x)$-well 2-placement $\sigma'$ such that
 $dist(z,\sigma'(z))\leq 3$. Thus, we can apply Claim 1 on $z$, a
 contradiction.
Hence, $m=2$ and so $T$ is  isomorphic to one of the trees in Fig.
2, and so, by Lemma 3.7, there
 exists a $(T,x)$-well 2-placement.
\\\\
 \textbf{Case 2.} $d(x)>1$.\\
If $x$ or any of its neighbors is a father of at least two leaves,
say $\alpha_1$ and $\alpha_2$, then let
$T'=T-\{\alpha_1,\alpha_2\}$. $T'$ is a star, since otherwise there
exists a $(T',x)$-well 2-placement, say $\sigma'$, such that
$dist(v,\sigma'(v))\leq 3$ for every $v\in \{ x\cup N(x)\}$ and so
we can  apply Claim  2 on the father of $\alpha_1$ and $\alpha_2$, a
contradiction. Hence, $T$ is isomorphic to $T_A$ or $T_B$ in Fig. 1,
 and so, by Lemma 3.6, there exists a $(T,x)$-well 2-placement. Else,
 suppose  that neither $x$ nor any  of its neighbors is a father of at least two
 leaves.
 If there exists a set of leaves, say $\{\alpha_1,...,\alpha_m\}$, $\;m\geq 2$,
 such that all of the leaves have the same father, say $\beta$, with $d(\beta)=m+1$, then
 let $T'=T- \{\alpha_1,...,\alpha_m\}$. Note that $T'$ is a non star tree since
   neither $x$ nor any  of its neighbors is a father of at least two
 leaves. Hence, there exists a $(T',x)$-well 2-placement, say
 $\sigma_x$, such that $dist(\beta,\sigma_x(\beta))\leq 4$ since
 $\beta$ is an end vertex in $T'$. Then
 $\sigma=\sigma_x(\alpha_1...\;\alpha_m)$ is a
 $(T,x)$-well 2-placement. Otherwise, suppose that the set of leaves
 $\{\alpha_1,...,\alpha_m\}$ doesn't exist. Since $T$ is not a path and all the previous cases
 are not satisfied then  there exists $ y \in N(x)$ such that $T_{(x,y)}$ is a non star tree.
 There exists no neighbor $ y$ of $x$ such that $T_{(x,y)}$   and $T_{(y,x)}$ are non star
 trees, since otherwise there exists a $(T_{(x,y)},x)$-well 2-placement, and so we can apply Claim 1 on $x$.
If there exists $y \in N(x)$ such that $T_{(x,y)}$ is a non star
tree and $T_{(y,x)}$ is a path of length two, then  all the neighbor
trees of $x$ are neighbor $F$-trees. And since $T$ is not a path
then $d(x)>2$. Let $\{y_1,...,y_m\}$, $m\geq 2$, be the neighbors of
$x$ distinct from $y$, $T_0$ be the connected component containing
$x$ in $T-\{xy_i;\;i=1,...,m\}$. Then $T_0$ is a path of length
three and so there exists, by Lemma 3.1, a $(T_0,x)$-well
2-placement. Hence, we can apply Claim 2 on $x$, a contradiction.
Thus, there exists no neighbor $y$ of $x$ such that $T_{(x,y)}$ is a
non star tree and $T_{(y,x)}$ is a path of length two. If there
exists $y \in N(x)$ such that $T_{(x,y)}$ is a non star tree and
$T_{(y,x)}$ is a path of length one, then if $d(x)=2$, let $y_1$ be
the neighbor of $x$ distinct from $y$. $T_{(y_1,x)}$ is not a
 neighbor $F$-tree of $x$ since $T$ is not a path. Let
 $\{b_1,...,b_r\}$, $r\geq 1$, be the neighbors of $y_1$ distinct from
 $x$ and let $T_0$ be the connected component containing $x$ in
 $T-\{y_1b_i;\;i=1,..., r\}$. Then $T_0$ is a path of length three, and so
  there exists, by Lemma 3.1, a $(T_0,x)$-well path 2-placement $\sigma_0$ such
   that $dis(y,\sigma_0(y))\leq 3$.
  Thus, $y_1$ has a unique neighbor $F$-tree distinct from
  $T_{(x,y_1)}$, since otherwise we can apply either Claim 1 or Claim 2 on $y_1$. Suppose that $T_{(b_1,y_1)}$ is that tree and let $T'$
  be the connected component containing $x$ in
  $T-\{y_1b_i;\;i=2,..., r\}$. Hence, $T'$ is a path of length at least four and at most six,
   and so, by Lemma 3.1, there exists a $(T',x)$-well path
  2-placement, say $\sigma'$, such that $dist(y,\sigma'(y ))\leq 3$.
  Thus, we can apply Claim 1 on $y_1$, a contradiction. Hence, $d(x)>2$, and so all the neighbor trees of $x$
are
 paths of length one with at most one is a path of length zero.
  Hence, by Lemma 3.9, there exists a
 $(T,x)$-well 2-placement.
  Otherwise, suppose that there exists no $y\in N(x)$ such that $T_{(x,y)}$
is a non star tree and $T_{(y,x)}$ is a path of length one. Then,
there exists $y \in N(x)$ such that $T_{(x,y)}$ is a non star tree,
$T_{(y,x)}$ is a path of length zero and $d(x)=2$. Let
$N(x)=\{y,y'\}$. If $d(y')=2$, then let $a$ be the neighbor of $y'$
distinct from $x$ and let $\{a_1,...,a_k\},\;k\geq 1$, be the
neighbors of $a$ distinct from $y'$. Let $T_0$ be the connected
component containing $x$ in $T-\{aa_i:\;i=1,...,k\}$, then $T_0$ is
a path of length three, and so, by Lemma 3.1, there exists a
$(T_0,x)$-well path 2-placement $\sigma_0$ such that
$dist(a,\sigma_0(a))\leq 3 $. Thus, $a$ has a unique neighbor
$F$-tree distinct from
 $T_{(y',a)}$, since otherwise we can apply either Claim 1 or Claim 2 on $a$. Suppose that $T_{(a_1,a)}$ is that tree and let $T'$ be
 the connected component containing $x$ in $T-\{aa_i;\; i=2,...,k\}$. Then $T'$ is
 a path of length at least four and at most six,
 and so, by Lemma 3.1, there exists a $(T',x)$-well path 2-placement, say
 $\sigma'$, such that $dist(a,\sigma'(a))\leq 3$. Thus, we can  apply Claim 1 on $a$, a contradiction.
 Hence, $d(y')>2$.
 let $\{y'_1,..,y'_l\},\; l\geq 2$, be the neighbors of $y'$ distinct from $x$.
If $y'$ has a non neighbor $F$-tree then suppose that
$T_{(y'_1,y')}$ is that tree.  Note that $T_{(y',y'_1)}$ is a non
star tree and so there exists a $(T_{(y',y'_1)},x)$-well 2-placement
say $\sigma'$ such that $dist(y',\sigma'(y'))\leq 3$. Thus, Claim 1
is applied on $y'$, a contradiction.
  Hence, all the neighbor trees of $y'$ are
  neighbor $F$-trees. If $d(y')>3$, then let $T'$  be the connected component containing $x$ in
$T-\{y'y'_i:\; i=2,...,l\}$. Since $T'$ is a non star tree, then
there exists a $(T',x)$-well 2-placement and so we can apply Claim 2
on $y'$, a contradiction. Thus,  $d(y')=3$ and $T$ is isomorphic to
one of the trees in Fig. 2, and so, by Lemma 3.7, there exists a
$(T,x)$-well
  2-placement. $\square$ \\\\
\textbf{Proof of Corollary 1.3}.\\
Let $T_0=T -\{\alpha_1, ..., \alpha_{m_T}\}$, where $\{\alpha_1,
..., \alpha_{m_T} \}$ is the maximal set of leaves that can be
removed from $T$ in such a way that the obtained tree is a non star
one. Since $T_0$ is a non star tree then there exists a $(T_0,
x)$-well 2-placement for some $x$ in $T_0$. We define a packing of
$T$ into $T^6$, say $\sigma$, as follows:\begin{center}
 $\sigma(v)=\begin{cases}
 \sigma'& \text { if $v\in V(T')$}\\
 \alpha_i & \text { if $v=\alpha_i$ for $i=1,..., m_T$}\\
 \end{cases}$
 \end{center}
 Label $\alpha_i$ by $i$, for $i=1,..., m_T$. Let $r$ be the number of
cycles of $\sigma$ and let $\sigma_1,...,\sigma_r$ be those cycles.
Remark that $r\geq \lceil \frac{n-m_t}{5}\rceil$. Label the vertices
of each cycle $\sigma_i$ by $m_T+i$ for $i=1,...,r$. Hence, we
obtain an $(m_T+r)$-labeled packing of $T$ into $T^6$ and so
$w^6(T)\geq
m_T+\lceil \frac{n-m_t}{5}\rceil$. $\square$\\\\
\textbf{Proof of Theorem 1.8}.\\ The proof is by induction on the
order of $T$. For $n=4$ there is only one non star tree, which is
$P_4$. By lemma 3.2, there exists  a $(T,v)$-good 2-placement for
every $ v \in V(P_4)$. Suppose that the theorem holds for $n'<n$,
$n\geq 5$, and let $T$ be a non star tree of order $n$. Let $x$ be a
vertex of $T$ such that $x$ is not  a bad vertex. If $T$ is a path,
then, by Theorem 3.2, there exists a $(T,x)$-good path 2-placement.
Else, we
will study two cases concerning the degree of $x$:\\\\
\textbf{Case 1.} $d(x)=1$.\\
 Let $y$ be the father of $x$. If $y$
is a father of another leaf $\alpha$, then let $T' = T-\{ x\}$. Note
that $T'$ is a non star tree since otherwise $T$ is a star. Hence,
there exists a $(T', \alpha)$-good 2-placement, say $\sigma'$. The
2-placement $\sigma$
 defined as follows:
\begin{center}
$\sigma(v)=\begin{cases}
 \sigma'(v) &\text{if $v\in V(T')-\{\alpha\}$}\\
 \sigma'(\alpha) & \text {if $v=x$}\\
 x & \text{if $v=\alpha$}\\
 \end{cases}$
 \end{center} is a $(T,x)$-good 2-placement.
 Otherwise,
suppose that $y$ is the father of the leaf $x$ only. If there exists
a set of leaves, say $\{\alpha_1, ...,\alpha_m\}, \; m\geq 2$, such
that all of these leaves have the same father, say $\beta$, with
$d(\beta) = m + 1$, then let $T' = T-\{\alpha_1, ..., \alpha_m\}$.
If $T'$ is a non star tree, then there exists a $(T',x)$-good
2-placement, say $\sigma_x$, such that
$dist(\beta,\sigma_x(\beta))\leq 4$ since $\beta$ is a leaf in $T'$.
Hence, $\sigma=\sigma_x(\alpha_1...\;\alpha_m)$ is a $(T,x)$-good
2-placement. And if $T'$ is a star then $T$ is isomorphic to $T_A$
or $T_C$ in Fig. 1, and so, by Lemma 3.6, there exists a $(T,
x)$-good 2-placement. Otherwise, suppose that the set $\{\alpha_1,
..., \alpha_m\} $ doesn't exist.  Hence, we can remark that for any
edge $ab$ in $E(T)$, $T_{(a,b)}$ can be either a neighbor $F$-tree
of $b$ or a non star tree. Let $z$ be the first vertex away from $x$
such that $d(z) > 2$. $z$ exists since $T$ is not a path. Let $z'$
be the nearest neighbor of $z$ to $x$ and let $\{z_1, ..., z_m\}$,
$m\geq 2$, be the neighbors of $z$ distinct from $z'$. Note that
$T_{(z',z)}$ is a path and $T_{(z,z')}$ is a non star tree, since
otherwise the set of leaves described above exists. If $T_{(z',z)}$
is a path of length at least three, then let $\sigma_x$ be  a
$(T_{(z',z)},x)$-good path 2-placement and let $\sigma_z$ be a
$(T_{(z,z')},z)$-good 2-placement if $z$ is not a bad vertex in
$T_{(z,z')}$, and if it is, then let $\sigma_z$ be a
$(T_{(z,z')},z'')$-good 2-placement, where $z''$ is a neighbor of
$z$ in $T_{(z,z')}$. Since $dist(z',\sigma_x(z'))\leq 2$, then the
2-placement $\sigma$ defined as follows:
 \begin{center}
 $\sigma(v)=\begin{cases}
 \sigma_x(v) & \text { if $v\in V(T_{(z',z)})$}\\
 \sigma_z(v) & \text { if $v\in V(T_{(z,z')})$}\\
 \end{cases}$
 \end{center} is a $(T,x)$-good 2-placement. Otherwise, $T_{(z',z)}$
 is a path of length at most two. If
$T_{(z',z)}$ is a path of length two, then let $\sigma'$ be a
$(T_{(z',y)},z')$-good 2-placement.
  The 2-placement $\sigma$  defined as follows:
\begin{center}
 $\sigma(v)=\begin{cases}
 \sigma'(v) & \text { if $v\in V(T)-\{x,y,z'\}$}\\
 x & \text { if $v=z'$}\\
 y & \text { if $v=x$}\\
 \sigma'(z') & \text { if $v=y$}\\
 \end{cases}$
 \end{center} is a $(T,x)$-good 2-placement.
Else, if
 $T_{(z',z)}$ is a path of length one,
 that is, $x$ is a neighbor of $z'$, then if $z$ is not a bad vertex in $T_{(z,z')}$,  let $\sigma'$ be a $(T_{(z,z')},z)$-good
 2-placement. The 2-placement  $\sigma$  defined as follows:
\begin{center}
 $\sigma(v)=\begin{cases}
 \sigma'(v) & \text { if $v\in V(T)-\{x,z,z'\}$}\\
 x & \text { if $v=z$}\\
 z' & \text { if $v=x$}\\
 \sigma'(z) & \text { if $v=z'$}\\
 \end{cases}$
 \end{center} is a $(T,x)$-good 2-placement.
And if $z$ is a bad vertex in $T_{(z,z')}$ then $T$ is isomorphic
 to $T_B$ in Fig. 6, and so, by Lemma 3.12, there exists a $(T,x)$-good
 2-placement.
 Otherwise,
  $T_{(z',z)}$ is a path of length zero, that is, $x$ is a neighbor
  of $z$. If there exists a neighbor of $z$, say $z_1$, such that
$T_{(z_1,z)}$ is not a neighbor $F$-tree of $z$, then let
$\sigma_{z_1}$ be a $(T_{(z_1,z)},z_1)$-good 2-placement if $z_1$ is
not a bad vertex in $T_{(z_1,z)}$, and if it is, then let
$\sigma_{z_1}$ be a $(T_{(z_1,z)},w_1)$-good 2-placement, where
$w_1$ is a neighbor of $z_1$ in $T_{(z_1,z)}$.  $T_{(z,z_1)}$ is a
non star tree since $z$ is the father of the leaf $x$ only, then
there exists a $(T_{(z,z_1)},x)$-good 2-placement, say $\sigma_x$,
such that $dist(z,\sigma_x(z))\leq 2$. We define a $(T,x)$-good
2-placement $\sigma$ as follows:
\begin{center}
 $\sigma(v)=\begin{cases}
 \sigma_{z_1}(v) & \text { if $v\in V(T_{(z_1,z)})$}\\
 \sigma_x(v) & \text { if $v\in V(T_{(z,z_1)})$}\\
 \end{cases}$
 \end{center}
 Otherwise, suppose that all  the neighbor trees of $z$ are neighbor $F$-trees.
 Then, by Lemma 3.10, there exists a $(T,x)$-good
2-placement.\\\\
\textbf{Case 2.} $d(x)>1$.\\
If $x$ or any of its neighbors is a father of at least two leaves
distinct from $x$, say $\alpha_1$ and $\alpha_2$, then let
$T'=T-\{\alpha_1,\alpha_2\}$. If $T'$ is a non star tree such that
$x$ is not a bad vertex in $T'$ then let $\sigma_x$ be a
$(T',x)$-good 2-placement. Then
$\sigma=\sigma_x(\alpha_1\;\alpha_2)$ is a $(T,x)$-good 2-placement.
 If $x$ is a bad vertex in $T'$ then $T$ is isomorphic to $T_A$ or $T_C$ in Fig. 5, and so, by Lemma 3.11, there exists a
 $(T,x)$-good 2-placement. And if $T'$ is a star then $T$ is isomorphic either to $T_B$ in Fig. 5 or
 to $T_A$ in Fig. 6, and so, by Lemma 3.11 and Lemma 3.12, there exists a $(T,x)$-good 2-placement.
 Otherwise,  $x$ and each of its neighbors is the father of at most one leaf.
  If there exists a set of leaves, say $\{\alpha_1,...,\alpha_m\}$, $m\geq
2$, such that all of the leaves have the same father, say $\beta$,
with $d(\beta)=m+1$,  then let $T'=T- \{\alpha_1,...,\alpha_m\}$.
Note that $T'$ is a non star tree since neither $x$ nor any of its
neighbors is a father of at least two leaves. If  $x$ is not a bad
vertex in $T'$, then there exists a $(T',x)$-good 2-placement, say
$\sigma_x$, such that $dist(\beta,\sigma_x(\beta))\leq 4$ since
$\beta$ is a leaf in $T'$. Thus
$\sigma=\sigma_x(\alpha_1...\;\alpha_m)$ is a $(T,x)$-good
2-placement.
 And if $x$ is a bad vertex in $T'$, then $T$ is isomorphic to  $T_D$ in Fig.
 5,
  and so, by Lemma 3.11, there exists a $(T,x)$-good 2-placement.
 Otherwise, suppose that  the set of leaves  $\{\alpha_1,...,\alpha_m\}$
doesn't exist in $T$. Since $T$ is not a path, neither $x$ nor any
of its neighbors is the father of at least two leaves and the set of
leaves having the same father doesn't exist, then there exists $y\in
N(x)$ such that $T_{(x,y)}$ is a non star tree. Whenever $x$ is a
bad vertex in $T_{(x,y)}$,  let $x_1$ and
 $x_2$ be the neighbors of $x$ in $T_{(x,y)}$ and $y_1$ and $y_2$ be
 that of $x_1$ and $x_2$, respectively. If there exists $ y \in N(x)$ such that $T_{(x,y)}$  and $T_{(y,x)}$ are non star
 trees, then let $\sigma_y$  be a $(T_{(y,x)},y)$-good 2-placement
 if $y$ is not a bad vertex in $T_{(y,x)}$, and   if it is, then let $\sigma_y$ be a $(T_{(y,x)},y')$-good
 2-placement,
  where $y'$ is a neighbor of $y$ in $T_{(y,x)}$. If $x$ is not a bad vertex in $T_{(x,y)}$, then let
$\sigma_x$ be a $(T_{(x,y)},x)$-good 2-placement. The 2-placement
$\sigma$  defined as follows:
 \begin{center}
 $\sigma(v)=\begin{cases}
 \sigma_x(v) & \text { if $v\in V(T_{(x,y)})$}\\
 \sigma_y(v) & \text { if $v\in V(T_{(y,x)})$}\\
 \end{cases}$
 \end{center} is a $(T,x)$-good 2-placement.
 And if $x$ is a bad vertex in $T_{(x,y)}$, then let $T'$ be the
 connected component containing $x$ in
 $T-\{xx_1,xx_2\}$ and let $\sigma_x$ be a $(T',x)$-good
 2-placement. Thus $\sigma=\sigma_x(x_2\;y_2\;x_1\;y_1)$ is a
 $(T,x)$-good 2-placement.
 Otherwise, if   there exists $ y \in N(x)$ such that $T_{(x,y)}$ is a non star tree and $T_{(y,x)}$ is a path of length zero, then
if $x$ is not a bad vertex in $T_{(x,y)}$,  let $\sigma_x$ be a
$(T_{(x,y)},x)$-good 2-placement.
 The 2-placement $\sigma$ defined such that:
 \begin{center}
 $\sigma(v)=\begin{cases}
 \sigma_x(v) & \text { if $v\in V(T_{(x,y)})-\{x\}$}\\
 y & \text { if $v=x$}\\
 \sigma_x(x)& \text { if $v=y$}\\
 \end{cases}$
 \end{center}is a $(T,x)$-good 2-placement.
And if $x$ is a bad vertex in $T_{(x,y)}$ then
$\sigma=(x\;y\;x_1\;y_1\;x_2\;y_2)$ is a $(T,x)$-good 2-placement.
Else, if  there exists $ y \in N(x)$ such that $T_{(x,y)}$ is a non
star tree and $T_{(y,x)}$ is a path of length two, then if $x$ is
not a bad vertex in $T_{(x,y)}$, let $\sigma_x$ be a
$(T_{(x,y)},x)$-good 2-placement and let $T_{(y,x)}=yzw$. The
2-placement $\sigma$ defined such that:
 \begin{center}
 $\sigma(v)=\begin{cases}
 \sigma_x(v) & \text { if $v\in V(T_{(x,y)})-\{x\}$}\\
 y & \text { if $v=x$}\\
 \sigma_x(x)& \text { if $v=z$}\\
 w & \text { if $v=y$}\\
 z & \text { if $v=w$}\\
 \end{cases}$
 \end{center} is a $(T,x)$-good 2-placement.
And  if $x$ is a bad vertex in $T_{(x,y)}$, then
$\sigma=(x\;y\;w\;z)(x_1\;y_1\;x_2\;y_2)$ is a $(T,x)$-good
2-placement. Else,   there exists $ y \in N(x)$ such that
$T_{(x,y)}$ is a non star tree and  $T_{(y,x)}$ is a path of length
one. If $d(x)=2$, then let $N(x)=\{y_1,y_2\}$. Suppose that
$T_{(y_1,x)}=y_1x_1$ and $T_{(x,y_1)}$ is a non star tree.
 Let $\{a_1,...,a_m\}$, $m\geq 1$, be the neighbors of $y_2$ distinct from $x$.
 If $y_2$ has a non neighbor $F$-tree, then suppose that $T_{(a_1,y_2)}$
 is that tree and
 let $\sigma_{a_1}$ be a $(T_{(a_1,y_2)},a_1)$-good 2-placement if $a_1$ is not a bad vertex in
 $T_{(a_1,y_2)}$, and if it is, then let $\sigma_{a_1}$ be a $(T_{(a_1,y_2)},b)$-good 2-placement, where $b$ is a neighbor
 of $a_1$ distinct from $y_2$. If $x$ is not a bad vertex in $T_{(y_2,a_1)}$, then let $\sigma_x$ be a $(T_{(y_2,a_1)},x)$-good 2-placement.
 Finally, the 2-placement $\sigma$  defined as follows:
 \begin{center}
 $\sigma(v) = \begin{cases}
  \sigma_x(v) & \text{if $v\in V(T_{(y_2,a_1)})$} \\
 \sigma_{a_1}(v) & \text{if $v\in V(T_{(a_1,y_2)})$}\\
\end{cases}$
\end{center} is a $(T,x)$-good
2-placement. If $x$ is a bad vertex in $T{(y_2,a_1)}$, then $y_2$
has only two neighbors distinct from $x$ such that $T_{(a_2,y_2)}$
is the vertex $a_2$. Let $\{b_1,...,b_l\}$, $l\geq 1$, be the
neighbors of $a_1$ distinct from $y_2$, $T'$ be the connected
component containing $x$ in $T-\{a_1b_i;\; i=1,...,l\}$ and let
$\sigma'=(x_1\;x\;y_2\;y_1)(a_1\;a_2)$. Whenever $T_{(b_i,a_1)}$,
$1\leq i\leq m$, is not a neighbor $F$-tree of $a_1$, let $\sigma_i$
be a $(T_{(b_i,a_1)}, b_i)$-good 2-placement if $b_i$ is not a bad
vertex in $T_{(b_i,a_1)}$, and if it is, then let $\sigma_i$ be a
$(T_{(b_i,a_1)},w_i)$-good 2-placement, where $w_i$ is a neighbor of
$b_i$ in $T_{(b_i,a_1)}$. If all the neighbor trees of $a_1$ are non
neighbor $F$-trees, then the
 2-placement $\sigma$  defined as follows:
 \begin{center}
 $\sigma(v) = \begin{cases}
  \sigma'(v) & \text{if $v\in V(T')$} \\
 \sigma_{i}(v) & \text{if $v\in V(T_{(b_i,a_1)})$ for$\; i=1,...,l$}\\
\end{cases}$
\end{center} is a $(T,x)$-good
2-placement. If $a_1$ has at least two neighbor $F$-trees, then
suppose that $T_{(b_i,a_1)}$ is a  neighbor $F$-tree of $a_1$ for
$i=1,...,p$, where $2\leq p\leq l$. By Corollary 3.1, there exists a
$(T'',x)$-good 2-placement $\sigma''$ such that
$\sigma''(v)=\sigma'(v)$ for every $ v$ of $T'$, where $T''$ is the
connected component containing $x$ in $T-\{a_1b_i;\; i=p+1,...,l\}$.
If $T''=T$, then $\sigma''$ is a $(T,x)$-good 2-placement.
Otherwise, a $(T,x)$-good 2-placement $\sigma$ is defined a s
follows:
 \begin{center}
 $\sigma(v) = \begin{cases}
  \sigma''(v) & \text{if $v\in V(T'')$} \\
 \sigma_{i}(v) & \text{if $v\in V(T_{(b_i,a_1)})$ for$\; i=p+1,...,l$}\\
\end{cases}$
\end{center}
Finally, if $a_1$ has a unique neighbor $F$-tree, then suppose that
$T_{(b_1,a_1)}$ is that tree. Let $T'$ be the connected component
containing $x$ in $T-\{a_1b_i;\; i=2,...,l\}$, then $T'$ is
isomorphic to one of the trees, $T_E$, $T_F$ or $T_G$, in Fig. 5,
and so, by Lemma 3.11, there exists a $(T,x)$-good 2-placement
$\sigma'$ such that $dist(a_1,\sigma'(a_1))\leq 2$. If $l=1$, then
$\sigma'$ is a $(T,x)$-good 2-placement. Else, a $(T,x)$-good
2-placement $\sigma$ is defined as follows:
 \begin{center}
 $\sigma(v) = \begin{cases}
  \sigma'(v) & \text{if $v\in V(T')$} \\
 \sigma_{i}(v) & \text{if $v\in V(T_{(b_i,a_1)})$ for$\;i=2,...,l$}\\
\end{cases}$
\end{center}
Now, suppose that all the neighbor trees of $y_2$ are neighbor
$F$-trees, then $d(y_2)\geq 3$ since $T$ is not a path. Let $T_0$ be
the connected component containing $x$ in $T- \{y_2a_i;\;i=1,...,
m\}$ and $\sigma_0=(x\;y_2\;y_1\;x_1)$. Then, by Corollary 3.1,
there exists a $(T,x)$-good 2-placement since
$dist(y_2,\sigma_0(y_2))=2$ and all the neighbor  trees of $y_2$
 are neighbor $F$-trees.
 Finally, if $d(x)>2$, then each neighbor tree of $x$ is a path of length one. Let $N(x)=\{y_1,...,y_r\}$, $r> 2$,
  and let $T_{(y_i,x)}=y_i\;x_i,$ for $i=1,..., r$. If $d(x)>4$,
  then let $T'$ be the connected component containing $x$ in
 $T-\{xy_i; i=4,..., r\}$. Since $T'$ is a non star tree
 and $x$ is not a bad vertex in $T'$, then there exists a
 $(T,x)$-good 2-placement. Since $T_{(y_i,x)}$ are neighbor $F$-trees of $x$ for $i=4,...,r$, then, by Corollary 3.1,
 there exists a $(T',x)$-good 2-placement. Otherwise,  that is, $d(x)<5$,
 then $\sigma=(x\;y_3\;y_2\;x_2\;x_3\;y_1\;x_1)$ is a $(T,x)$-good
 2-placement if $d(x)=3$, and $\sigma'=(x\;y_3\;y_2\;x_2\;x_3\;y_4\;x_4\;y_1\;x_1)$ is a $(T,x)$-good
 2-placement if $d(x)=4$. $\square$\\\\
\textbf{Proof of Corollary 1.4}.\\ Let $T'=T-
\{\alpha_1,...,\alpha_{m_T}\}$,  where
$\{\alpha_1,...,\alpha_{m_T}\}$ is the maximal set of leaves that
can be removed from $T$ in such a way that the obtained tree is a
non star one. Since $T'$ is a non star tree, then there exists a
$(T',x)$-good 2-placement, say $\sigma_x$, where $x$ is any  non bad
vertex of $T'$. We define a packing $\sigma$ of $T$ into $T^5$  as
follows:
\begin{center}
 $\sigma(v) = \begin{cases}
  \sigma_x(v) & \text{if $v\in V(T')$} \\
 \alpha_i & \text{if $v=\alpha_i$ for $ i=1,... m_T$}\\
\end{cases}$
\end{center}
Label $\alpha_i$ by $i$ for $i=1,..., m_T$ and label all the
vertices in $T'$ by $m_T+1$. Hence, we obtain an $(m_T+1)$-labeled
packing of $T$ into $T^5$, and so $w^5(T)\geq m_T+1$. $\square$

 
 \end{document}